\documentclass[12pt,a4paper]{article}
\usepackage{calligra}
\DeclareMathAlphabet{\mathcalligra}{T1}{calligra}{m}{n}
\usepackage[english]{babel}
\usepackage[T1]{fontenc}
\usepackage{graphicx}
\usepackage{epstopdf}
\usepackage{epsfig,psfrag}
\usepackage{amsmath, amsfonts, amsthm,amssymb}
\usepackage{latexsym}  
\usepackage{hyperref}
\usepackage{color}
\usepackage{empheq}
\usepackage{bbold} 
\newtheorem{definition}{Definition}[section]
\newtheorem{lemma}{Lemma}[section]

\newtheorem{theorem}{Theorem}[section]

\newtheorem{remark}{Remark}[section]

\newtheorem{proposition}{Proposition}[section]

\numberwithin{equation}{section}
\newcommand{\ad}{{\rm ad}}
\def\ii{\mathrm i}
\def\Z{\mathbb Z}
\def\N{\mathbb N}
\def\cN{\mathcal N}
\def\R{\mathbb R}
\def\C{\mathbb C}

\def\e{\varepsilon}
\def\T{\mathbb T}

\def\SS{\mathcal S}

\def\OO{\mathcal O}
\def\HH{\mathcal H}
\def\I{\mathcal I}
\def\a{\alpha}
\def\b{\beta}

\def\s{\sigma}

\def\vt{\vartheta}
\def\e{{\varepsilon}}
\def\fr{\mathtt{r}}

\def\muu{{\omega}}
\def\laa{{\lambda}}
\newcommand{\re}[1]{\cP\!_{#1}}

\newcommand{\al}{{\alpha}}
\newcommand{\bt}{{\beta}}
\newcommand{\Psii}{{\Phi}}
\newcommand{\cA}{{\mathcal A}}

\newcommand{\cF}{{\mathcal F}}

\newcommand{\cH}{{\mathcal H}}

\newcommand{\cO}{{\mathcal O}}
\newcommand{\cP}{{\mathcal P}}

\newcommand{\cR}{{\mathcal R}}
\newcommand{\cS}{{\mathcal S}}
\newcommand{\cT}{{\mathcal T}}

\newcommand{\abs}[1]{\big| #1 \big|}
\newcommand{\ti}{{\mathtt{i}}}

\newcommand{\td}{{\mathtt{d}}}

\def\be{\begin{equation}}
\def\ee{\end{equation}}
\title{KAM for beating solutions of the quintic NLS}
\author{E. Haus*, M. Procesi**
\vspace{2mm} 
\\ \small 
$^*$ Dipartimento di Matematica e Applicazioni ``R. Caccioppoli''\\ \small Universit\`a degli Studi di Napoli ``Federico II'', Napoli, I-80126, Italy
\vspace{2mm}
\\ \small 
$^{**}$ Dipartimento di Matematica ``G. Castelnuovo''\\ \small Universit\`a di Roma ``La Sapienza'', Roma, I-00185, Italy
\vspace{2mm}
\\ \small 
E-mail:  \texttt{emanuele.haus@unina.it}, \texttt{mprocesi@mat.uniroma1.it}}

\begin{document}
\maketitle
\abstract{We consider the  nonlinear Schr\"odinger equation of degree five on the circle $\mathbb S^1= \R / 2\pi$. We prove the existence of quasi-periodic solutions which bifurcate from ``resonant'' solutions (studied in \cite{GT}) of the system obtained by truncating the Hamiltonian after one step of Birkhoff normal form, exhibiting recurrent exchange of energy between some Fourier modes. The existence of these quasi-periodic solutions is a purely nonlinear effect.}

\section{Introduction}

We consider the quintic  NLS on the circle
\begin{equation}\label{nls}
-i\partial_tu+\partial_{xx} u=\sigma |u|^4u\ , \qquad\qquad u=u(t,x) \qquad (t,x)\in\mathbb R \times \mathbb T \qquad \mathbb T:= \mathbb R / (2\pi)
\end{equation}
where $\s=\pm 1$.  This  is an infinite dimensional dynamical system with Hamiltonian
\begin{equation}\label{nn}
\HH=\int_{\mathbb{T}}|\nabla u|^2+\frac{\s}{3}\int_{\mathbb{T}}|u|^6
\end{equation}
having the mass (the $L^2$ norm) and the momentum
\begin{equation}
L=\int_{\mathbb{T}^2}|u|^2\ , \quad M= \int_{\mathbb{T}^2} {\rm Im }(u\cdot \nabla u) 
\end{equation}
as constants of motion.
Classical results imply that, for small values of the mass, equation \eqref{nls} is globally well-posed in $H^s$ with $s\geq1$. This is a consequence of the local well posedness result in \cite{Bou93}, combined with the conservation law \eqref{nn} (if the mass is small, the Hamiltonian controls the $H^1$-norm also for $\s=-1$, via the Gagliardo-Nirenberg inequality). However, one expects solutions of \eqref{nls} to exhibit a rich variety of qualitative behaviors.
Even close to zero, very little is known on the qualitative structure of typical solutions and the literature is confined to the study of special solutions which exhibit interesting features.
A fruitful approach  is to take the dynamical systems point of view and  to apply the powerful tools of {\em singular perturbation theory}, such as KAM theory, Birkhoff Normal Form, Arnold diffusion, first developed in order to study finite-dimensional systems.

A first elementary feature is that NLS equation has an elliptic fixed point at $u=0$. However such fixed point is  completely resonant  (i.e. all the linear solutions are periodic). 
It is well known, already in a finite dimensional setting, that such resonances may produce hyperbolic  tori, secondary tori  and energy transfer phenomena where the solutions of the non--linear system differ drastically from the ones of the linear system. 

The dynamical systems approach was first applied  in the early '90s to model PDEs  on $[0,\pi]$ with Dirichlet  boundary conditions in order to prove existence of quasi-periodic solutions and finite dimensional reducible KAM tori (see Definition \ref{def.qp} and \ref{reducible}). Among the vast literature we mention \cite{W1}, \cite{Kuk87}, \cite{KP} and \cite{Bou94},\cite{CW}, \cite{CY} for the case of periodic boundary conditions.  We note that the above papers all take advantage of some simplifying assumptions and hence they do not apply to our setting.  The first results on reducible KAM tori for NLS equations were proved in \cite{EK} for equations with external parameters, see also \cite{PX}.  Reducible KAM tori for \eqref{nls} were then proved in \cite{LY} and \cite{GY}, see also \cite{PP},\cite{PP16} and \cite{V}. Differently from the (integrable) cubic case, it can be seen that for appropriate choices of the initial data such solutions have a finite number of  linearly unstable directions. Then it is quite natural to wonder whether one may prove the existence of an orbit which starts close to the unstable manifold of one  solution and in very long time  arrives close to the stable manifold of a different solution, thus  exhibiting a large drift in the action variables.
This approach was first proposed in \cite{CKSTT} in order to study the growth of Sobolev norms for the cubic NLS on $\T^2$ (see \cite{GK} for an estimate of the time) and then generalized to any analytic NLS on $\T^2$ in \cite{HP} and \cite{GHP}. Unfortunately the approach proposed in this papers fails for equations on the circle such as \eqref{nls}, see Appendix C of \cite{HP} for a discussion of the problem. The main  point is that even if it is still true that unstable solutions exist one is not able to prove full energy transfer.

In this paper, we prove existence of quasi-periodic solutions of \eqref{nls} which exhibit {\em recurrent energy exchange} between some Fourier modes. These solutions are genuinely non-linear: they do not bifurcate from linear solutions, which do not feature any transfer of energy between Fourier modes. Moreover, these quasi-periodic solutions correspond to \emph{secondary} KAM tori (see for instance \cite{BC} for a discussion about secondary KAM tori in the finite-dimensional context), 
 since they are 4-dimensional, but homotopically equivalent to 3-dimensional tori in the phase space. 
To the best of our knowledge, this is the first result establishing existence of secondary KAM tori for PDEs.
%
%
\paragraph{Main results and strategy of the paper.}
We start by defining the main objects on which our result relies.
\begin{definition}\label{def.qp}
Given a rationally independent\footnote{i.e. a vector $\omega\in \R^n$ such that for all $k\in \Z^n\setminus \{0\}$ one has $\omega\cdot k\neq 0$.}  $\omega\in \R^n$, a quasi-periodic solution of frequency $\omega$ is a torus embedding $\mathfrak I: \T^n\ni\varphi \mapsto u(\varphi, x)$ 
such that $u(\omega t, x)$ is a classical solution for \eqref{nls}.
\end{definition}
 \begin{definition}[reducible KAM torus] \label{reducible}
 A quasi-periodic solution of frequency $\omega$  is a  reducible KAM torus if  (in appropriate complex symplectic variables $\varphi,y,z,\bar z$) it is expressed  by   the equations $y=z=0$ and moreover the associated Hamiltonian vector field  $X_H$  restricted to the torus  is $\sum_{i=1}^n\omega(\xi)\partial_{x_i}$,  while $X_H$  linearized at the torus is block-diagonal in the normal variables with $x$-independent block matrices   of uniformly bounded dimension.
\end{definition}Having fixed our notation, the main result of the paper is the following.
\begin{theorem}\label{thm:main}
There exist $\varepsilon_0>0$ and a compact domain $\OO_0\subset\R^4$ such that, for all $0<\varepsilon<\varepsilon_0$, there exists $\OO_\varepsilon \subset \OO_0$, satisfying $|\OO_\varepsilon|/|\OO_0|\to1$ as $\varepsilon\to0$, and a locally invertible Lipschitz mapping $\muu^{\infty}:\OO_\e\to\R^4$,  with the following property.
Fix $k\in\Z$ and let $\mathcal S:=\{-2,-1,1,2\}$. For all $\xi=(\xi_0,\xi_1,\xi_2,\xi_3)\in\OO_\varepsilon$, equation \eqref{nls} admits a quasi-periodic
solution\footnote{By the notation $O(\e^{2/3^{-}})$ in equation \eqref{la soluzione}, we denote a remainder $\mathfrak r_{\e,k}(\xi;t,x)$ belonging to the space $\ell^{a,p}$ introduced in \eqref{scale} for some $a>0,p>1/2$ fixed once and for all, satisfying $\| \mathfrak r_{\e,k} \|_{a,p} \leq C_\delta \e^{\frac23 - \delta}$ for all $\delta>0$.}
\begin{equation}\label{la soluzione}
u_{\e,k}(\xi;t,x)={\varepsilon}^{1/4} \Big(\sum_{j\in \mathcal S} v_j(\xi;\muu^\infty(\xi)t)e^{i (j+k)x} e^{i (jk  +k^2 )t} + O(\e^{2/3^{-}})\Big)
\end{equation}
of frequency $\muu^\infty(\xi)=(0,1,1,4)+  O(\e)$, such that the $v_j$ satisfy
\begin{align}
|v_j(\xi;\varphi_0,\varphi_1,\varphi_2,\varphi_3)|^2 = & \ \mathtt f_j (\xi;\varphi_0) \notag \\
\mathtt f_{-2}(\xi;0)\leq \frac12 \quad \mathtt f_{1}(\xi;0)\leq 1 \qquad & \ \mathtt f_{-1}(\xi;0)\geq 3 \quad \mathtt f_{2}(\xi;0)\geq \frac32 \label{battimenti davvero} \\
\mathtt f_{-2}(\xi;\pi)\geq \frac32 \quad \mathtt f_{1}(\xi;\pi)\geq 3 \qquad & \ \mathtt f_{-1}(\xi;\pi)\leq 1 \quad \mathtt f_{2}(\xi;\pi)\leq \frac12 \ . \notag
\end{align}
Moreover, for all $\varphi_0\in\T,\xi\in\OO_\e$ one has
\begin{equation}\label{xi123}
\begin{cases}
\mathtt f_1(\xi;\varphi_0) + 2 \mathtt f_2(\xi;\varphi_0) = \xi_1 \\
\mathtt f_{-1}(\xi;\varphi_0) - 2 \mathtt f_2(\xi;\varphi_0) = \xi_2 \\
\mathtt f_{-2}(\xi;\varphi_0) + \mathtt f_2(\xi;\varphi_0) = \xi_3
\end{cases}
\end{equation}
while
\begin{equation}\label{xi0}
\mathtt f_{2}(\cdot,\xi_1,\xi_2,\xi_3;0) - \mathtt f_{2}(\cdot,\xi_1,\xi_2,\xi_3;\pi) \text{ is an increasing function of } \xi_0.
\end{equation}
The quasi-periodic solutions $u_{\e,k }$ are analytic reducible KAM tori.
\end{theorem}
The proof of this Theorem is performed essentially in two steps, formally described in Propositions \ref{formanorm} and \ref{KAM}. We now give an informal presentation of the strategy. From now on, we shall systematically work on the Hamiltonian \eqref{nn}, moreover since the sign $\s$ will play no role in our reasoning, we shall for simplicity set it to one in the following.

\smallskip

It is convenient to study the Hamiltonian \eqref{nn} in the Fourier coefficients
$
u(t,x)=\sum_{j\in\Z} u_j(t) e^{ijx}
$ where the sequence $\{u_j\}_{j\in \Z}\in \ell^{a,p}$. 
We rescale $u \rightsquigarrow \varepsilon^{1/4} u$, so that the Hamiltonian $\HH$ assumes the form
\begin{equation}\label{NLS1}
\HH=\sum_{j\in\Z} j^2|u_j|^2+\varepsilon\sum_{j_1+j_2+j_3=j_4+j_5+j_6} u_{j_1}u_{j_2}u_{j_3}\bar u_{j_4}\bar u_{j_5}\bar u_{j_6} 
\end{equation}
with the constants of motion
\begin{equation}\label{costanti0}
\mathbb L=\sum_{j\in\Z} |u_j|^2\,,\qquad \mathbb M= \sum_{j\in \Z} j |u_j|^2.
\end{equation}

\begin{remark}\label{wmbw}
	For all $k\in\Z$, equation \eqref{nls} is invariant with respect to the transformation
	\begin{equation}\label{traslo}
	v_j(t) := u_{j+k} (t) e^{-it(k^2+2jk)}.
	\end{equation}
	Since the actions satisfy $|v_j|^2=|u_{j+k}|^2$, this means that for any given solution $u(t,x)$ of \eqref{nls}, there exists another solution $v(t,x)$ whose actions have the same dynamics, up to a translation by $k$ of the Fourier support.
\end{remark}
It is well known that the dynamics of \eqref{NLS1} is controlled for finite but long times by the corresponding Birkhoff Hamiltonian, containing only the terms in \eqref{NLS1} which Poisson commute with the quadratic part $H^{(2)}:= \sum_{j\in \Z} j^2 |u_j|^2$, see Subsection \ref{birkhof}. In the case of Hamiltonian \ref{NLS1} we get 
$$
H_{\rm Birk}= \sum_{j\in\Z} j^2|u_j|^2+ \varepsilon\!\!\!\!\!\!\sum_{\substack{ j_1+j_2+j_3=j_4+j_5+j_6 \\  j_1^2+j_2^2+j_3^2=j_4^2+j_5^2+j_6^2}} \!\!\!\!\!\! u_{j_1}u_{j_2}u_{j_3}\bar u_{j_4}\bar u_{j_5}\bar u_{j_6}
$$
which was studied in detail in \cite{GT} and \cite{HT}. A simple remark, proved for instance in \cite{PP} Propositions 1 and 2, is that while the dynamics of $H_{\rm Birk}$ may be very complicated still one may produce many subspaces 
 which are invariant for the  dynamics of $H_{\rm Birk}$ and such that the restricted dynamics  is integrable. 
 
  More precisely given a
  set $\mathcal S\subset \Z$ we define the subspace
  $$
  U_\mathcal S:= \{u\in L^2(\T):  \quad u(x)= \sum_{j\in \mathcal S} u_j
  e^{\ii j x}\}\,,
  $$
  and consider the following definitions.
  \begin{definition}[Complete]\label{completeness}
  	We say that a set $\SS\subset\Z$ is {\em complete} if $U_\SS$ is invariant for
  	the dynamics of $H_{\rm Res}$.
  \end{definition}
  \begin{definition}[Action preserving]\label{integrable}
  	A complete set $\SS\subset\Z$ is said to be {\em action preserving} if all the
  	actions $|u_j|^2$ with $j\in \SS$ are constants of motion for the dynamics of
  	$H_{\rm Res}$ restricted to $U_\SS$.
  \end{definition}
  The conditions under which a given set $\SS$ is complete or action preserving
  can be rephrased more explicitly by using the structure of $H_{\rm Birk}$.
  \begin{definition}[Resonance]\label{risuona} Given a sextuple
  	$(j_1,\dots,j_{6})\in (\Z)^{6}$ we say that it is a resonance 
  	if
  	\begin{equation}\label{resonance}
  	\begin{cases}
  	j_1+j_2+j_3=j_4+j_5+j_6 \\  j_1^2+j_2^2+j_3^2=j_4^2+j_5^2+j_6^2
  	\end{cases}
  	\end{equation}
  \end{definition}
   With this Definitions  one easily sees that generic choices of $\SS$  lead to complete and action preserving subspaces. Indeed $\SS$ is complete if and only if for any quintuple
   $(j_1,\dots,j_{5})\in \SS^{5}$ there does not  exist any  $k\in
   \Z\setminus\SS$ such that $(j_1,\dots,j_{5},k)$ is a resonance.
   Similarly $\SS$ is action preserving if all resonances $(j_1,\dots,j_{6})\in
   \SS^{6}$ are {\em trivial}, namely there exists a permutation such that
   $(j_1,j_2,j_3)=(j_{4},j_5,j_{6})$.  
   
   Note that the tori constructed in \cite{GY}, \cite{LY}  and \cite{PP16} are  all essentially supported on such generic sets $\SS$. In this paper we study instead a simple choice of set $\SS$ which is {\em complete but not action preserving}. Following \cite{GT}, we choose  $\mathcal S:=\{-2,-1,1,2\}$ and its translations, see Remark \ref{traslo}. 
   As shown in \cite{GT}, $H_{\rm Birk}$ restricted to $\SS$ is in fact still integrable (it has four degrees of freedom and  three constants of motion, see Formulas \ref{ristretto} and \ref{aggiungi}) however it has a non-trivial structure. To study this system we restrict to a domain which is the product of a 4-cube and a 4-torus, such that all the variables in the 4-cube are constants of motion for the dynamics of $H^{(2)}$.  If we take into account also the resonant terms of degree $6$ we get a different qualitative behavior.
    Indeed after symmetry reduction  it is a Hamiltonian system on a cylinder  with one stable  and {one} unstable fixed point with its separatrix, see Figure \ref{vaffa2}.
Namely there is one   stable   and  one unstable 3-torus coming from the resonant structure.  We concentrate our attention on the new set of quasi periodic solutions, which are families of stable 4-tori around the new stable 3-torus. The structure of these solutions as well as their action angle variables is studied in section \ref{quattro}.
A key point is that these tori satisfy the so-called {\em twist condition}, i.e. the determinant of the Hessian of $H_{\rm Birk}$ restricted to $\SS$ is not identically zero, see Lemma \ref{twistlemma}.
In the paper \cite{GT}, these tori were  used to prove the existence of {\em beating} solutions of the \eqref{nls} for long but finite times. This follows from the fact that the dynamics of $H_{\rm Birk}$ approximates the true dynamics of \eqref{NLS1} for long times. 

We wish to prove the persistence of {\em Cantor families} of such quasi periodic solutions for infinite time, as well as their linear stability by applying a KAM scheme. A main point naturally is to control the behavior of the infinitely many {\em normal directions}, i.e. the Fourier coefficients $u_j$  with $j\neq \pm 1,\pm 2$. As is well known, a crucial preliminary point is to study the reducibility of the Birkhoff system linearized  at the quasi-periodic solutions. This amounts to removing the angle dependence and diagonalizing an infinite dimensional quadratic Hamitonian. 

As happens already in the case of solutions bifurcating from linear tori (see \cite{PP}), this reduction is not trivial and not of a perturbative nature but must be constructed essentially by hand. In our case we proceed as follows:
 
\begin{enumerate}
 \item We perform a phase shift \eqref{phase1} which removes the dependence on all the angles except one. We are left with a quadratic Hamiltonian which depends on one angle and  is diagonal except for a $4\times4$ block.
 \item We  apply a phase shift \eqref{phase2} which removes the dependence on the last angle from all the terms of the Hamiltonian except for the non-diagonal  $4\times4$ block.   
 \item To this last {\em finite dimensional} block we apply {\em Floquet Theorem} \ref{Floquet.complex},  in order to remove the dependence on the angle. Finally we diagonalize it by using the  standard theory of quadratic Hamiltonians.
\end{enumerate}
This proves the first item in Proposition \ref{formanorm}.  In order to prove the second item we must have some control of the eigenvalues of our quadratic Hamiltonian in the normal directions, we discuss this in subsection \ref{melni}. The third item follows by a careful study of our changes of variables.
Once we have proved Proposition \ref{formanorm}, the NLS Hamiltonian essentially fits all the hypotheses of a KAM scheme for a PDE on the circle. In Proposition \ref{KAM} we give a brief description of such algorithms, in order to make the paper self contained. 

\section{Functional setting and main Propositions}\label{funset}
\subsection{Functional setting}
Given a possibly empty set $\mathcal S\subset\Z$ and setting $n= | \mathcal S|$ we consider the scale of complex Hilbert spaces:
\begin{equation}\label{scale}
{\bf{\ell}}^{a,p}_\mathcal S:=\{ \{u_k \}_{k\in \Z\setminus\SS}\;\big\vert\;|u_0|^2+\sum_{k\in \Z\setminus\mathcal S} |u_k|^2e^{2 a |k|} |k|^{2p}:=||u||_{a,p}^2 \le \infty \},\end{equation}
where $a>0,\ p>1/2$ are fixed once and for all. When $ \mathcal S = \emptyset $ we denote $ \ell^{a,p} := \ell^{a,p}_{ \mathcal S}  $. We consider the direct product 
\begin{equation}\label{E}
E:= \C^n \times \C^n \times \ell^{a, p}_{ \mathcal S} \times \ell^{a, p}_{ \mathcal S} 
\end{equation}
We endow  the space $ E $  with the $(s,r)$-weighted norm 
\begin{equation}\label{normaEsr}
\mathtt v =  (\varphi,y,z,\bar z) \in E \, , \quad
\|\mathtt v\|_E:=\|\mathtt v\|_{E,s,r}= \frac{|\varphi|_\infty}{s} + \frac{|y|_1}{r^2}
 +\frac{\|z\|_{a,p}}{r}+\frac{\|\bar z\|_{a,p}}{r}
\end{equation}
where, $ 0 < s, r < 1 $, and 
$ |\varphi|_\infty := \max_{h =1, \ldots, n} |\varphi_h| $,
$ |y|_1 := \sum_{h=1}^n |y_h| $.
We shall also use the notations
$$
z_j^+ = z_j \, , \quad z_j^- = \bar z_j \, .
$$
We remark that $E$ is a symplectic space w.r.t. the form $dy \wedge d\varphi + \ii \sum_{k}d u_k\wedge d \bar u_k$.
We also define the toroidal domain
\be\label{Dsr}
D (s,r) := \T^n_s \times D(r) := \T^n_s \times B_{r^2} \times B_r
\times  B_r \subset E
\ee
where $ D(r):= B_{r^2} \times B_r \times B_r $,
\be\label{seconda}
 \T^n_s := \Big\{ \varphi \in \C^n \, : \, \max_{h=1, \ldots, n} |{\rm Im} \, \varphi_h | < s  
 \Big\} \, , \ \
 B_{r^2}  := \Big\{ y \in \C^n \, : \, |y |_1 < r^2 \Big\}
\ee
and $ B_{r} \subset \ell^{a,p}_{\mathcal S} $ is the open ball of radius $ r $
centered at zero. We think of $\T^n$ as the $ n $-dimensional torus
$ \T^n := 2\pi \R^n/\Z^n $.
\begin{remark}
If $ n = 0 $  then $ D(s,r) 
 \equiv B_r\times B_r \subset
\ell^{a,p}\times \ell^{a,p}$.
\end{remark}
For a vector field  $X:  D(s,r)\to E$, described by the formal Taylor expansion:
$$
 X = \sum_{\nu,i,\a,\b} X_{\nu,i,\a,\b}^{(\mathtt v)} e^{\ii \nu \cdot \varphi} y^i z^\a \bar z^\b \partial_{\mathtt v}\,, \quad \mathtt v= \varphi,y,z,\bar z$$   
 
 we define the {\em majorant} and its {\em norm}\footnote{The different weights in formula \eqref{normaEsr} ensure that, if $| X|_{s,r}$ is sufficiently small,  then $X$ generates a close--to--identity  change of variables from $D(s/2,r/2)\to D(s,r)$.}:
\begin{eqnarray}\label{normadueA}
MX & :=& \sum_{\nu,i,\a,\b} |X_{\nu,i,\a,\b}^{(\mathtt v)}| e^{s|\nu|} y^i z^\a \bar z^\b \partial_{\mathtt v}\,, \quad \mathtt v= \varphi,y,z,\bar z\nonumber\\
| X |_{s,r} & := &
\sup_{(y,z, \bar z) \in D(s,r)} \| M X \|_{E,s,r} \,.  \end{eqnarray}

 In our algorithm we deal with Hamiltonian vector fields  which depend in a Lipschitz  way on some parameters $\xi\in \R^n$ in a compact set $ \mathcal O \subset \R^n$. To handle this dependence one introduces weighted  norms.   Given  $X:  \OO \times D(s,r)\to E$ we set
 \begin{equation}
 \label{weno}   | X(\xi)|^\gamma_{s,r,\mathcal O}:=  \sup_{\xi\in \mathcal O}| X(\xi)|_{s,r} +\gamma \sup_{\eta\neq \xi\in \mathcal O}\frac{| X(\xi)-X(\eta)|_{s,r}}{|\xi-\eta|}
 \end{equation}  where $\gamma $ is a  parameter.  Sometimes when $\mathcal O$ is understood we just  write $ | X|^\gamma_{s,r,\mathcal O}= | X|^\gamma_{s,r}$.
 \begin{definition}
  We say that a Hamiltonian function is {\em $M$-regular} if its associated vector field is $M$-analytic i.e. has finite \eqref{normadueA} norm. 
  We define the norm of such $H$ as
  \begin{equation}\label{hasidin}
  |H|^\gamma_{s,r}:= |X_H|^\gamma_{s,r}\,.
  \end{equation}
 We denote the space of functions with finite $|\cdot |^\gamma_{s,r}$ norm as $\cA_{s,r,\cO}\equiv \cA_{s,r}$.
 \end{definition}
\subsection{Main Propositions}
We fix once and for all\footnote{by Remark \ref{wmbw} we could trivially shift for any $k$, to $\cS=\{k-2,k-1,k+1,k+2\}$. }
$
\cS:=\{-2,-1,1,2\}
$ and $n=4$.  As explained in the introduction we first obtain a {\em normal form} result for our Hamiltonian \eqref{NLS1}.
\begin{proposition}\label{formanorm}
 There exist  $\e_0,s_0,r_0>0$  and a compact domain $\mathcal O_0\subset \mathbb R^4 $ of  positive measure such that the following holds:
 
 for all $\e<\e_0,s<s_0,r<r_0 $ there exists an analytic function
 $$
\Psi:\, \OO_0\times D(s,r) \to \ell^{a,p} \,,
 $$
 such that  for all  $\xi\in \OO_0$ the map $ \Psi(\xi; \cdot): D(s,r) \to \ell^{a,p}$  is symplectic and satisfies the following properties:

 \begin{itemize}
\item[$(i)$] Define  $v(\xi;\varphi) = \{ v_j (\xi;\varphi) \}_{j\in\Z} := \Psi (\xi; \varphi, y = 0, w = 0)$. Then, for all $j\in\SS$, equations \eqref{battimenti davvero}, \eqref{xi123}, \eqref{xi0} hold.

\item[$(ii)$] The NLS Hamiltonian $\mathcal H$ defined in \eqref{NLS1} in the new variables becomes:
\begin{equation}
\mathcal H\circ \Psi:=\HH_{\mathtt{NLS}}= \muu(\xi)\cdot y +  \sum_{j\neq\pm1,\pm2,} \Omega_j |w_j|^2 + \mathcal R
 \,,\quad \Omega_j= (j^2 + \varepsilon \Theta_j(\xi)) \,,
 \end{equation}
 where 
 $$
 \muu(\xi)= (0,1,1,4) +\e \laa(\xi)
 $$
 with $\xi\to \laa(\xi)$ an invertible analytic map  and $\Theta_j$ is chosen in a finite list (in fact with 5 elements)
 of distinct analytic functions.
 
\item[$(iii)$]  For all $\xi\in \OO_0$, we have the  bounds:
 $$
 |\laa|+ |\partial_\xi \laa|<M_0\,,\quad |\partial_\laa \xi|<L_0\,, \quad|\Theta_j|+ |\partial_\xi \Theta_j|<M_0 \,, \forall j\neq \pm 1,\pm 2
 $$
 $$
 |\laa\cdot \ell + \Theta_i \pm \Theta_j|>\alpha_0, \quad \forall (\ell,i,j)\neq (0,i,i): |\ell|< 4 M_0L_0
 $$
 $$
 |\mathcal R|_{s,r,\OO_0}^{1} \le  {\rm C} (\e^2r^{-2} +\e r )\,,
 $$
 
 \item[$(iv)$]  $\HH_{\mathtt{ NLS}}$ has the constants of motion:
  $$
  \mathbb L= y_1+y_2 + y_3 + \sum_{j\neq \pm 1,\pm2} |w_j|^2\,,$$
  $$
 \mathbb  M =y_1-y_2-2 y_3 +\sum_{j=3,4} j (|w_j|^2+ |w_{-j}|^2)+  \sum_{j\neq \pm 1,\pm2,\pm 3,\pm 4} j |w_j|^2
  $$
  \end{itemize}
\end{proposition}
\begin{proof} The proof of this Proposition is deferred  to subsection \ref{prova1}, since it requires all the structure developed in Sections \ref{birkhof}, \ref{quattro}. 
\end{proof}
  \begin{proposition}\label{KAM}
Consider the  Hamiltonian $\cH_{\tt NLS}$ in Proposition  \ref{formanorm} and fix  $r=\e^{1/3}$. There exist  $\e_\star,s_\star, \gamma_\star, {\mathtt C}_\star >0$  and,  for all $ \e \in(0,\e_\star)$ and  $\gamma\in ({\mathtt C}_\star \e^{1/3},\gamma_\star) $,
a compact domain $\mathcal O_\infty= \OO_\infty(\gamma,\e ) \subseteq \OO_0 $ of  positive measure with $|\OO_0\setminus\OO_\infty|\sim \gamma$, such that for all $s<s_0$ there exists a Lipschitz family of analytic affine symplectic maps 
  $$ \OO_\infty \ni \xi \mapsto \Psii_\infty(\xi; \cdot): \;D\left(\frac s 4,\frac r 4\right) \to D(s,r)$$ with the following properties.
  \begin{itemize}
  \item[$(i)$] $ \Psii_\infty$ is $O(\e^{1/3}\gamma^{-1})$ close to the identity, i.e.
  $$
  \| \Psii_\infty(\xi; \mathtt v) - \mathtt v\|_{E,s/4,r/4} \le  \e^{1/3}\gamma^{-1} \qquad \forall \mathtt v \in D\left( \frac s 4, \frac r 4 \right)
  $$
\item[$(ii)$] One has $$
  \cH_{\tt NLS} \circ \Psii_{\infty}:= \muu^{\infty}\cdot y +  \sum_{j\neq\pm1,\pm2,} \Omega^\infty_j |w_j|^2 + \mathcal R^\infty
  $$
  where
  $$
  X_{\cR^\infty}\vert_{y=0 \atop w=0}=0\,,\quad d_{y,w} X_{\cR^\infty}\vert_{y=0 \atop w=0}=0 \,, 
  $$
and 
 $$
 \muu^\infty(\xi)= (0,1,1,4) +\e \laa^\infty(\xi) \,,\quad \Omega_j^{\infty}(\xi)= j^2+ \e \Theta^{\infty}_j
 $$
 with $\xi\to \laa^\infty(\xi)$ an invertible Lipschitz  map  $|\laa^{\infty}-\laa|\sim \e^{1/3}\gamma^{-1}$, $|\Theta_j^{\infty}-\Theta_j|\sim \e^{1/3}\gamma^{-1}$.
  \end{itemize}
  
  \end{proposition}
  \begin{proof} The proof of this proposition is deferred to Section \ref{sec:KAM}.
  \end{proof}
  
  \subsection{Proof of Theorem \ref{thm:main}}
  We choose $\OO_0$ as in Proposition \ref{formanorm} and $\e_0 > 0$ as the smaller between $\e_0$ in Proposition \ref{formanorm} and $\e_\star$ in Proposition \ref{KAM}, same for $s_0$. We choose $\gamma = (-\log\e)^{-1}$ so that the hypotheses of Proposition \ref{KAM} are fulfilled for $\e$ small enough. Then, for all $\xi \in \OO_\infty(\gamma, \e)$ given by Proposition \ref{KAM}, the 4-dimensional torus $\Psi \circ \Phi_\infty (\xi; \varphi, y = 0, w = 0)$ is a reducible KAM torus, by Proposition \ref{KAM} $(ii)$. Equations \eqref{la soluzione}-\eqref{xi0} hold true by Proposition \ref{formanorm} $(i)$, by Proposition \ref{KAM} $(i)$, and by the definition of the norm $\| \cdot \|_{E,s,r}$ in \eqref{normaEsr}. This gives the thesis for $k=0$.
  
The fact that the thesis holds for all $k \in \Z$ is a trivial consequence of Remark \ref{wmbw}.

  \section{Birkhoff Normal form}\label{birkhof}
 We study the Hamiltonian \eqref{NLS1} which, following the definitions of Section \ref{funset}, is an analytic function $:B_R\times B_R\to \C$ with maximal degree $6$.  
We apply a step of Birkhoff normal form (cf. \cite{Bo3},\cite{Bo2},\cite{bamb}),  by which we cancel all the  terms  of degree $6$ which do not Poisson commute with   $$ H^{(2)}:= \sum_{j\in \Z} j^2 |u_j|^2.$$ 
  This is done by applying  a well known analytic change of variables, with generating function 
 \begin{equation}\label{birk1}
  F:=\e \sum_{\alpha,\beta\in (\Z)^\N: |\alpha|=|\beta|=3\atop {\sum_k (\alpha_k-\beta_k)k=0\,,\;\sum_k (\alpha_k-\beta_k)|k|^2\neq 0}} \hskip-10pt\binom{3}{\alpha}\binom{3}{\beta}\frac{u^\alpha\bar u^\beta}{\sum_k (\alpha_k-\beta_k)|k|^2}\,.
  \end{equation}
  We denote the change of variables by $\Psi_{\rm Birk}:=e^{ad(F)}$ and notice that, for $\e$ small enough, it is well defined and analytic:    $   B_{1}\times B_{1} \to B_{2} \times B_{2}$.

   By construction $\Psi_{\rm Birk}$  brings (\ref{NLS1}) to  the form
$$
\HH_{\mathtt {NLS}}=H_{\rm Birk}+ \varepsilon^2 R_{\geq10}=\sum_{j\in\Z} j^2|u_j|^2+ \varepsilon\!\!\!\!\!\!\sum_{\substack{ j_1+j_2+j_3=j_4+j_5+j_6 \\  j_1^2+j_2^2+j_3^2=j_4^2+j_5^2+j_6^2}} \!\!\!\!\!\! u_{j_1}u_{j_2}u_{j_3}\bar u_{j_4}\bar u_{j_5}\bar u_{j_6}+ \varepsilon^2 R_{\geq10}
$$
where the terms of degree $6$ (which we denote by $H_6$) are supported on  of {\em resonant sextuples}, see Definition \ref{resonance}. 
The remainder term $R_{\geq 10}$ is a M-regular analytic Hamiltonian with minimal degree $10$, defined for $u\in B_1$ and with norm
$
|R|_{1} 
$ controlled by some $\e$ independent constant.

Let us analyze in detail the term $H_6$ by studying the resonant sextuples. Among the resonant sextuples, there are the trivial ones corresponding to $\{j_1,j_2,j_3\}=\{j_4,j_5,j_6\}$. We call {\em action preserving} the trivial resonant sextuples (and the corresponding monomials in the Hamiltonian), while we call {\em effective} all the others.

We denote  $\SS=\{-2,-1,1,2\}\subset\Z$ as the set of {\em tangential sites}. The complementary set $\Z\setminus\SS$ is the set of {\em normal sites}. Whenever $j$ is a normal site, we denote $u_j=z_j$. We order the term $H_6$ according to the degree in the normal variables: we denote by $H_{6,k}$ the part of $H_6$ that is homogeneous of degree $k$ in the $z_j$'s so that
$$
H_6=\sum_{k=0}^6 H_{6,k}\ .
$$

With our choice of $\SS$, the terms $H_{6,k}$ with $k\leq2$ have a relatively simple form. In particular, the only integer solutions to \eqref{resonance} with all six elements belonging to $\SS$
\begin{equation}
\begin{cases}
(j_1,j_2,j_3)=(1,1,-2)\\
(j_4,j_5,j_6)=(-1,-1,2)
\end{cases}
\end{equation}
up to permutations of $\{j_1,j_2,j_3\}$, permutations of $\{j_4,j_5,j_6\}$ and exchange of $(j_1,j_2,j_3)$ and $(j_4,j_5,j_6)$. Moreover, there are no solutions to \eqref{resonance} with five elements in $\SS$ and one element outside $\SS$ (see Lemma 2.4 of \cite{GT}) so $\SS$ is complete. Finally, it is easy to verify that the only integer solutions to \eqref{resonance} with four elements in $\SS$ and two elements outside $\SS$ are of the form
\begin{equation}
\begin{cases}
(j_1,j_2,j_3)=(1,2,-3)\\
(j_4,j_5,j_6)=(-1,-2,3)
\end{cases}
\qquad\qquad
\begin{cases}
(j_1,j_2,j_3)=(2,2,-4)\\
(j_4,j_5,j_6)=(-2,-2,4)
\end{cases}
\end{equation}
up to permutations of $\{j_1,j_2,j_3\}$, permutations of $\{j_4,j_5,j_6\}$ and exchange of $(j_1,j_2,j_3)$ and $(j_4,j_5,j_6)$.

With the remarks above, we compute the terms $H_{6,k}$ with $k\leq2$ and obtain
$$H_{6,0}=H_{6,0}^{\rm ap}+H_{6,0}^{\rm eff}$$
$$
H_{6,0}^{\rm ap}=6\left(\sum_{j\in\SS}|u_j|^2\right)^3-9\left(\sum_{j\in\SS}|u_j|^2\right)\left(\sum_{j\in\SS}|u_j|^4\right)+4\left(\sum_{j\in\SS}|u_j|^6\right)
$$
$$
H_{6,0}^{\rm eff}=9(u_1^2 u_{-2} \bar u_{-1}^2 \bar u_2 + c.c.)
$$
$$H_{6,1}=0$$
$$H_{6,2}=H_{6,2}^{\rm ap}+H_{6,2}^{\rm eff}$$
$$
H_{6,2}^{\rm ap}=\left[18\left(\sum_{j\in\SS}|u_j|^2\right)^2-9\left(\sum_{j\in\SS}|u_j|^4\right)\right]\left(\sum_{j\notin\SS}|z_j|^2\right)$$
$$
H_{6,2}^{\rm eff}=36(u_{-1}u_{-2}\bar u_1 \bar u_2 z_3 \bar z_{-3} + c.c.) + 9 (u_{-2}^2  \bar u_2^2 z_4\bar z_{-4} + c.c.)\ .
$$
We push the terms of degree at least three in the normal sites to the remainder, namely we write
$$
\mathcal H=H_2+ \varepsilon H_{6,0}+ \varepsilon H_{6,2}+\mathcal R
$$
with
$$
\mathcal R= \varepsilon H_{6,3}+ \varepsilon H_{6,4}+ \varepsilon H_{6,5}+ \varepsilon H_{6,6}+ \varepsilon^2 R_{\geq10}\ .
$$
Note that the constants of motion remain unchanged after the Birkhoff change of variables.
\begin{remark}
Here and in the following we will denote by the same letter $\mathcal R$  ``perturbation terms'' that we are for the moment ignoring,  which we will  bound at the end of the procedure.  
\end{remark}
We have shown that the Birkhoff Hamiltonian restricted to the invariant subset $\mathcal S$ has the form:
\begin{align}\label{ristretto}
\sum_{j\in \SS} j^2 |u_j|^2+ &6\e (\sum_{j\in\SS}|u_j|^2)^3-9\e (\sum_{j\in\SS}|u_j|^2)(\sum_{j\in\SS}|u_j|^4)+4\e\sum_{j\in\SS}|u_j|^6\\
+& 9\e(u_1^2 u_{-2} \bar u_{-1}^2 \bar u_2 + \bar u_1^2 \bar u_{-2}  u_{-1}^2 u_2)\nonumber
\end{align}
with the constants of motion
\begin{equation}\label{aggiungi}
|u_1|^2+2 |u_2|^2\,,\quad |u_{-1}|^2-2 |u_2|^2\,,\quad |u_{-2}|^2+|u_2|^2
\end{equation}
For $j\in\SS$, we first pass to the {\em symplectic polar coordinates} $(\I,\phi)$ namely
\begin{equation}\label{act1}
u_j=\sqrt{\I_j}e^{i\phi_j}\qquad j\in\SS\ ,
\end{equation}
so that the symplectic form is now $d \I\wedge \phi +i d z\wedge d \bar z$. 
We work in the domain 
$$
 B_{s_0,r_0}:=
 \mathtt I_0 \times \T^4_{s_0}\times B_{r_0}\times B_{r_0}
$$
with $\mathtt I_0$ a compact domain  in $(0,\infty)^4$, so that  the change of variables $(\I,\phi, z,\bar z) \to u$ is well defined and analytic. One easily sees that, for $\e $ small enough,
 $\mathcal R$ and  $\partial_{\I_j} R$  with $j=1,\dots,4$, are  M-regular Hamiltonians and
 $$
\sup_{(\I,\phi, z,\bar z) \in B_{s_0,r_0}}\Big( |X_\cR^{(\phi)}|_\infty + |X_\cR^{(\I)}|_1 + r_0^{-1}\| X_\cR^{(z)}\|_{a,p} + r_0^{-1}\| X_\cR^{(\bar z)}\|_{a,p}\Big)\le {\rm C}(\e r_0 + \e^2r_0^{-1}).
 $$

Then we make the analytic symplectic change of coordinates $(J,\vartheta)\rightsquigarrow(\I,\phi)$ defined by
\begin{align}\label{act2}
& J_0= \I_2 \,,\quad J_1= \I_1+ 2 \I_2\,,\quad J_2= \I_{-1}-2 \I_{2} \,,\quad J_3= \I_{-2}+ \I_{2} \nonumber
\\
 & \vt_0=-2\phi_1-\phi_{-2}+2\phi_{-1}+\phi_2 \quad \vt_1=\phi_{1}
\quad \vt_2=\phi_{-1} \quad \vt_3=\phi_{-2}\ .
\end{align}
Note that now 
$J_2$  is  not necessarily positive.
We have
$$
L= J_1+J_2+J_3 +\sum_{j\neq \pm 1,\pm 2} |z_j|^2
$$
$$
M= J_1-J_2-2J_3 +\sum_{j\neq \pm 1,\pm 2}j |z_j|^2
$$
$$
H_2= J_1+J_2+4 J_3 +\sum_{j\neq \pm 1,\pm 2}j^2 |z_j|^2
$$

\begin{align*}
H_{6,0}^{\rm ap} = & 6(J_1+J_2+J_3)^3 -9 (J_1+J_2+J_3)(J_0^2 + (J_1 - 2 J_0)^2 +(J_2 +2 J_0)^2 + (J_3 - J_0)^2) \\
& +4 (J_0^3 + (J_1 - 2 J_0)^3 +(J_2 +2 J_0)^3 + (J_3 - J_0)^3)
\end{align*}
\begin{equation*}
H_{6,0}^{\rm eff} = 18 (J_1 - 2 J_0) (J_2 + 2 J_0) \sqrt{J_0(J_3-J_0)}\cos\vt_0
\end{equation*}
$$
H_{6,2}^{\rm ap}=\bigg(18(J_1 + J_2 + J_3)^2-9(J_0^2 + (J_1 - 2 J_0)^2 +(J_2 +2 J_0)^2 + (J_3 - J_0)^2)\bigg)\left(\sum_{j\notin\SS}|z_j|^2\right)
$$
\begin{align*}
H_{6,2}^{\rm eff}=\; & 72 \sqrt{J_0(J_1 - 2 J_0)(J_2 + 2 J_0)(J_3 - J_0)} {\rm Re}\bigg(z_3 \bar z_{-3} e^{\ii (-\vt_0 +3(\vt_2 -\vt_1))}\bigg)\\
& + 18 J_0(J_3-J_0)  {\rm Re} \bigg(z_4\bar z_{-4} e^{\ii (-2\vt_0 +4(\vt_2 -\vt_1))}\bigg)\ .
\end{align*}
We now write the Hamiltonian in a more compact notation (note that $(J,\vt)\in \R^4\times \T^4$):
$$
\HH_{\mathtt {NLS}}=J_1+J_2+ 4 J_3 + \varepsilon H^{\rm ap}_{6,0}(J)+ \varepsilon H^{\rm eff}_{6,0}(J,\vt_0)+ \sum_{j=3,4}\tilde H_j +\sum_{j\neq \pm 1,\pm 2,\pm 3,\pm 4} (j^2+\varepsilon f(J)) |z_j|^2+\mathcal R
$$
where
$$
\tilde H_j=  (j^2+ \varepsilon f(J)) (|z_j|^2+ |z_{-j}|^2)+ \varepsilon \mathcal U_j(J){\rm Re}(e^{\ii (n_j \vt_0+\ell_j (\vt_2-\vt_1))} z_j \bar z_{-j})
$$
with
\begin{eqnarray}
 f(J)&=& 18(J_1 + J_2 + J_3)^2-9(J_0^2 + (J_1 - 2 J_0)^2 +(J_2 +2 J_0)^2 + (J_3 - J_0)^2)\nonumber\\
 \mathcal U_3&=& 72 \sqrt{J_0(J_1 - 2 J_0)(J_2 + 2 J_0)(J_3 - J_0)}\,,\label{above}\\
  \mathcal U_4&=& 18 J_0(J_3-J_0)\nonumber 
\end{eqnarray}
and finally
$$
n_3= -1 \,,\quad \ell_3= 3 \,,\quad n_4= -2 \,,\quad \ell_4= 4.
$$
We first make a change of variables which removes the the phases in $\tilde H_j$.
We set
\begin{align}\label{phase1}
& w_{ -j}= e^{-\ii (n_j \vt_0+\ell_j (\vt_2-\vt_1))}z_{ -j} \,,j=3,4 \quad w_j=z_j, \ |j|\geq5\\
& p= J_0-  \sum_{j=3,4} n_j|z_{-j}|^2\,,\;\; q =\vt_0, \nonumber\\
& K_1= J_1- \sum_{j=3,4} \ell_j|z_{-j}|^2, \quad
K_2= J_2+ \sum_{j=3,4} \ell_j|z_{-j}|^2, \quad
K_3=J_3  \nonumber
\end{align}
and we get
$$
\mathbb L= K_1 + K_2 + K_3 +\sum_{j\neq \pm 1,\pm2}|w_j|^2\,,$$ 
$$\mathbb M= K_1-K_2-2K_3+  \sum_{j=3,4} j(|w_j|^2+|w_{-j}|^2)+\sum_{j\neq \pm 1,\pm2,\pm 3,\pm 4}j|w_j|^2
$$
Now the conjugate variables are $(p,q), (K,\vt)$.
Substituting we get the Hamiltonian
\begin{equation}\label{camel}
\HH_{\mathtt {NLS}}= \mathcalligra H\;(p,K,q)+ \sum_{j=3,4}H_j(p,K,q,w_{\pm j}) +\sum_{j\neq \pm 1,\pm 2,\pm 3,\pm 4} (j^2+ \varepsilon f(p,K)) |w_j|^2 +\mathcal R.
\end{equation}
Now $\mathcal R$ contains some new terms of degree at least four in $w_{\pm 3}, w_{\pm 4}$. Here 
$$
\mathcalligra H\;(p,K,q)= A(K,p)+ B(K,p)\cos q
$$
with
\begin{equation}\label{ABnormal}
A(K,p)=K_1+K_2+4K_3+\varepsilon H_{6,0}^{\rm ap}(p,K), \quad
B(K,p)\cos q=\varepsilon H_{6,0}^{\rm eff}(p,K,q)
\end{equation}
while
$$
H_j= (j^2+ \varepsilon f(p,K)) (|w_j|^2+ |w_{-j}|^2)+ \varepsilon\mathcal U_j(p,K){\rm Re}( w_j \bar w_{-j})+\mathtt V_j(p,K,q)|w_{-j}|^2 
$$
with $f,\mathcal U_j$ defined in \eqref{above} and
$$
\mathtt V_j= \varepsilon \mathcal V_j:=  \ell_j \bigg(\partial_{K_1}\mathcalligra H -\partial_{K_2}\mathcalligra H\;\; \bigg)
+ n_j \partial_p \mathcalligra H\;.$$
Note that all the changes of  variables are analytic provided that we require that $r_0$ is sufficiently small and we fix $|w|_{a,p}\le r$ so that $|z|_{a,p}\le r e^{6s}<r_0$. The bounds on $\cR$ and on its $\partial_K$ derivatives remain the same.

\section{The integrable invariant subspace}\label{quattro}

We now restrict to the invariant subspace $\{z_j=0\}$. Our Hamiltonian is of the form
$$
\mathcalligra H\;= A(p,K)+B(p,K)\cos q
$$
with $A,B$ defined in \eqref{ABnormal}. The  corresponding dynamical system:
\begin{eqnarray}\label{pqSyst}
\begin{cases}
\dot q = \partial_p A+ \partial_p B \cos q  \\
\dot p = B\sin q
\end{cases}
\end{eqnarray}
 has been  studied in detail in \cite{GT}, here we group some facts that we will need. Let us introduce the rescaled time $\tau=\varepsilon t$, so that the system \eqref{pqSyst} in the rescaled time is $\varepsilon$-independent.
\begin{lemma}
There exists a neighborhood ${\mathcal B} \subset \R^3$ of $K_\star=(4,0,2)$ such that the following holds:
\begin{itemize}
\item[(i)] there exists an $\varepsilon$-independent analytic function $\mathfrak p: \cal B \to \R$ with $\mathfrak p(K_\star)=1$ such that $q=0, p={\mathfrak p}(K)$ is a stable fixed point for \eqref{pqSyst};
\item[(ii)] there exists an $\varepsilon$-independent analytic function $\mathfrak P: \cal B \to \R$ with $\mathfrak P(K_\star)=1$ such that $q=\pi, p={\mathfrak P}(K)$ is an unstable fixed point for \eqref{pqSyst};
\item[(iii)] up to the $2\pi$-periodicity in $q$, these are the only fixed points of the system \eqref{pqSyst} and the phase portrait is qualitatively the same as in Figure \ref{vaffa2};
\item[(iv)] for $K=K_\star$, the two homoclinic connections linking the unstable fixed points to itself intersect the axis $q=0$ at two points with $p=p^{(1)}, p=p^{(2)}$, satisfying $p^{(1)}+p^{(2)}=2$ and
\begin{equation}\label{quantobatte}
p^{(1)} < \frac12, \qquad p^{(2)} > \frac32.
\end{equation}
\end{itemize}
\end{lemma}
\begin{proof} The system is $2\pi$-periodic in $q$, direct computation shows that for $K=K_\star=(4,0,2)$
$$
\mathcalligra H\;= 1308 +\varepsilon \bigg( -270 (p^2 + (2-p)^2) +36 (p^3 + (2-p)^3)  + 72 p^{3/2}(2-p)^{3/2}\cos q\bigg)
$$
there is a stable fixed point at $q=0,p=1$ (a non degenerate maximum for the Hamiltonian) corresponding to a periodic solution and a non degenerate unstable fixed point $q=\pi, p=1 $. 

Then by the Implicit Function Theorem we construct $\mathfrak p(K)$  by solving the equation
\begin{equation}\label{col piu}
\partial_p A+ \partial_p B =0
\end{equation}
for $p={\mathfrak p}(K)$ in a neighborhood of $K=K_\star, p=1$.
In the same way, since $B(p,K)\neq 0$ then we obtain $\mathfrak P$  by solving the equation
\begin{equation}\label{col meno}
\partial_p A- \partial_p B =0.
\end{equation}
Note that $\varepsilon$ factorizes in equations \eqref{col piu}, \eqref{col meno}.
The qualitative structure of the phase portrait follows by Morse theory.
Finally, formula \eqref{quantobatte} is obtained by direct computation (see the evaluation of $\kappa_\star$ in Section 4.1 of \cite{GT}).
\end{proof}

 \begin{figure}[!ht]
    \centering
   \begin{minipage}[b]{11cm}
   \centering
   {\psfrag{I}{$p$}
   \psfrag{a}[c]{$-\pi$}
   \psfrag{b}[c]{$\pi$}
   \psfrag{f}[r]{$q$}
   \includegraphics[width=11cm]{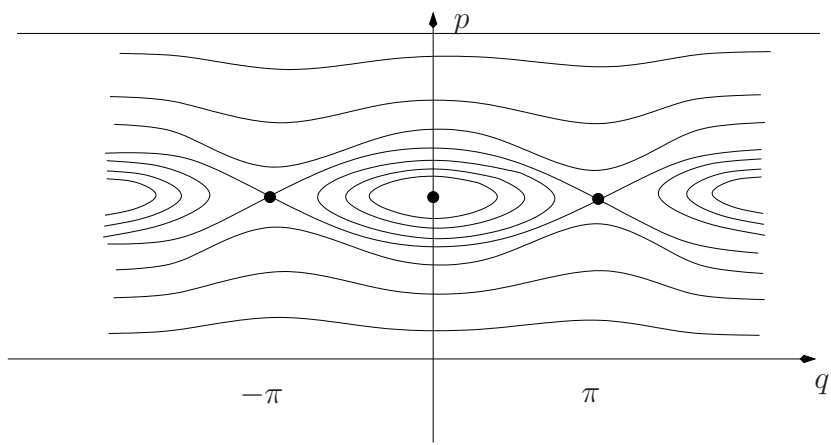}
   }
   \caption{\footnotesize{The phase portrait of $\mathcalligra H$\; at some value $K\in \cal B$.}}\label{vaffa2}
   \end{minipage}
   \end{figure}
%

By the previous Lemma for all $K\in \mathcal B$ we have an open domain delimited by the heteroclinic connections where the motion is oscillatory in $p,q$. Then in  all of this domain minus the  stable fixed point we can construct an $\e$-independent  symplectic change of variables which conjugates the system to action angles 
\begin{align}\label{integro}
&(p,K;q,\vt)\stackrel{\Psi}{\longrightarrow }(E,K;\psi,\varphi): \quad \mathcalligra H\;\circ \Psi = H(E,K,\e)
\\
& p=p(E,K,\varphi)\,,\quad q=q(E,K,\varphi)\,,\quad \vt= \psi+ \mathcal F(E,K,\varphi)\nonumber
\end{align}
note that $L,M$ do not depend on $p$ (nor on the angles) hence $L(K)\circ \Psi=L(K)$ and the same for $M$.
 %
%
%

We denote by $\mathcalligra D\;$ the (open) domain of definition of $H$.
\begin{lemma}\label{twistlemma}
 The Hamiltonian $H(E,K)$ is real-analytic in $\mathcalligra D\;$moreover there exists a proper algebraic hypersurface $\mathcal Z$ such that for all $E,K\in \mathcalligra D \setminus \mathcal Z$ the determinant of the Hessian matrix  $\partial_{E,K}^2 H(E,K)$ is not zero.
\end{lemma}
\begin{proof} We first note that the change of variables  is {\bf real} analytic in  $\mathcalligra D\;$, in particular  for each $K$ in a neighborhood of $K_\star$ the $p,q$ are analytic in $E$ in the whole of the region delimited by the two separatrices, excluding the fixed point.
We wish to compute the twist i.e. $\partial_{E,K}^2 H(E,K)$ and show that its determinant is a non-identically zero analytic function.
We verify this near the stable fixed point, where the computations are explicit.

By Vey's theorem the Birkhoff normal form at the fixed point converges, namely there exists a symplectic change of variables
$$\mathtt P= \mathtt P(p-{\mathfrak p}(K),q,K)\,,\quad \mathtt Q= \mathtt Q(p-{\mathfrak p}(K),q,K) \,,\quad \phi= \vt + \Phi(p-{\mathfrak p}(K),q,K)  $$ which conjugates the Hamiltonian $\mathcalligra H\;$ to $H(\mathtt P^2+\mathtt Q^2,K)$  in a neighborhood of $\mathtt P,\mathtt Q=0$. 
Note that  the variables $K$ play the r\^{o}le of parameters for $\mathcalligra H\;$. Hence we first  construct  symplectic changes of variables for $p,q$ (depending parametrically on $K$) at the end of the procedure we complete the symplectic change of variables by adjusting the angles conjugated to $K$.
   Since the change of variable is analytic and the actions are uniquely defined we must have that $\mathtt P^2+\mathtt Q^2=E$ and the Birkhoff map  gives an analytic extension to zero. So we may compute the $E,K$ derivatives of $H$ from  the Birkhoff normal form at $E=0,K=K_\star$.

We start by Taylor expanding in $p,q$ the Hamiltonian   at the point $({\mathfrak p}, 0)$ up to order four. We  denote $G= A+B$, see formula \eqref{ABnormal}. Given a function $f(K,p)$ we write\footnote{ Note that the application $f\mapsto f^\sharp$ does not commute with the derivatives. We denote $f^\sharp_K := (f_K)^\sharp $ and $\partial_K f^\sharp := \partial_K (f^\sharp)$, similarly for the derivatives with respect to $p$.} $f^{\sharp} (K):=f (K,\mathfrak p(K))$.  
\begin{align*}
\mathcalligra H\;\,(p,q)=\; &G^\sharp + \frac 12   G_{pp}^\sharp (p-{\mathfrak p})^2 -\frac 12  B^\sharp q^2 +\frac 1 6  G_{ppp}^\sharp(p-{\mathfrak p})^3 - \frac 12  B_p^\sharp (p-{\mathfrak p}) q^2\\
&+ \frac{1}{24}B^\sharp q^4- \frac14  B_{pp}^\sharp (p-{\mathfrak p})^2 q^2 + \frac{1}{24}G_{pppp}^\sharp (p-{\mathfrak p})^4 + \mathcal O\left((|q|+ |p-{\mathfrak p}|)^5\right).
\end{align*}
Note that $\e^{-1}(\mathcalligra H\, - G^\sharp)$ is $\e$-independent, while $G^\sharp$ is of order $O(1)$ in $\e$. It follows that the Birkhoff changes of coordinates that we shall perform are all $\e$-independent.
First we symmetrize the quadratic terms and translate the critical point to zero.
$$
P= \lambda (p-{\mathfrak p})\,,\quad Q= q \lambda^{-1} \,,\quad \lambda^4:= -\frac{G^\sharp_{pp} }{B^\sharp}
$$
(recall that $G_{pp}^\sharp<0$ and $B^\sharp>0$   in a neighborhood of $K_\star$) so,  setting $$\alpha^{(0)}= G^\sharp\,,\quad \alpha^{(2)}:= \sqrt{ - B^\sharp  G^\sharp_{pp}} \,,$$ we get 
$$
\mathcalligra H\;=\alpha^{(0)} - \frac 12  \alpha^{(2)} (P^2+Q^2) +\frac 1 6  G^\sharp_{ppp} \lambda^{-3} P^3 - \frac 1 2  B^\sharp_p \lambda P Q^2 +.
$$
$$
+ \frac{1}{24}B^\sharp \lambda^4 Q^4- \frac14 B^\sharp_{pp} P^2 Q^2 + \frac{1}{24}G_{pppp}^\sharp \lambda^{-4} P^4 +\mathcal O \left(( |P|+|Q|)^5\right)
$$
It is convenient to pass to complex notation in order to perform the Birkhoff normal form,
let us write the Hamiltonian (neglecting higher order terms) as
$$
\mathcalligra H= \alpha^{(0)} - \frac 12  \alpha^{(2)}  (P^2+Q^2)+ \sum_{h=3,4}\sum_{i+2j=h} \alpha^{(h)}_{i,j}P^i Q^{2j}
$$
where
$$
\alpha^{(3)}_{3,0}= \frac 1 6  G^\sharp_{ppp} \lambda^{-3} \,,\quad \alpha^{(3)}_{1,1}=  - \frac 1 2 B^\sharp_{p} \lambda\,,
$$
and
$$
\alpha^{(4)}_{4,0}= \frac{1}{24}G_{pppp}^\sharp \lambda^{-4}\,,\quad \alpha^{(4)}_{2,1}= - \frac {1}{4}  B^\sharp_{pp} \,,\quad \alpha^{(4)}_{0,2}=  \frac{1}{24}B^\sharp \lambda^4.
$$
we set
$\sqrt{2}  z= P+\mathrm i Q$, so that
$$
\mathcalligra H= \alpha^{(0)}- \alpha^{(2)}|z|^2+ \sum_{h=3,4}\sum_{i+2j=h} (-1)^j \frac{\alpha^{(h)}_{i,j}}{\sqrt{2}^h}(z+\bar z)^i(z-\bar z)^{2j}
$$
$$
=\alpha^{(0)}- \alpha^{(2)}|z|^2+ \sum_{h=3,4}\sum_{l+m=h}\beta^{(h)}_{l,m}z^l \bar z^m
$$
with
$$
\beta^{(l+m)}_{l,m}=  \frac{1}{\sqrt{2}^{l+m}}\sum_{i+2j=l+m} \sum_{\begin{subarray}{c} 0\leq a\leq i \\
0\leq b\leq 2j \\ a+b= l \end{subarray}} (-1)^j \alpha^{(l+m)}_{i,j} \binom{i}{a}\binom{2j}{b}\ .
$$
In particular we compute
$$
\beta^{(4)}_{2,2}
=\frac32(\alpha^{(4)}_{0,2}-\alpha^{(4)}_{2,1}+\alpha^{(4)}_{4,0})=\frac{1}{16} B^\sharp \lambda^4 + \frac38 B_{pp}^\sharp + \frac{1}{16} G_{pppp}^\sharp \lambda^{-4}\ .
$$
Direct computations show that $\beta_{l,m}^{(3)}(K_\star)=0$.

Now we need to remove the terms of order three: we perform a change of variables with generating function
$$F= \sum_{l+m=3} F_{l,m}(K)z^l \bar z^m $$ with $F(K_\star)=0$
which cancels the terms of degree three.
We have
\begin{equation}\label{change-ord3}
F_{l,m}= \frac{\beta^{(3)}_{l,m}}{i \alpha^{(2)} (l-m)}
\end{equation}
note that since $l+m=3$ then $l-m$ cannot be zero.

This gives the new Hamiltonian with new terms of order four
$$ 
\mathcalligra H= \alpha^{(0)}- \alpha^{(2)}|z|^2+ \sum_{l+m=4}\gamma^{(4)}_{l,m}z^l \bar z^m+ h.o.t.
$$
and we know that 
\begin{equation}\label{mizzi}\gamma^{(4)}_{l,m}(K_\star)= \beta^{(4)}_{l,m}(K_\star). \end{equation}
since $F(K_\star)=0$.
Now we remove the terms of order $4$ which are not in the kernel. By definition
the kernel is the part depending only on $|z|^2$ i.e.
$\gamma^{(4)}_{2,2}$.  This procedure may be repeated indefinitely and, by Vey's Theorem, it converges to
$H(|z|^2,K)$. By construction: 
$$
H(|z|^2,K)=  \alpha^{(0)}(K)- \alpha^{(2)}(K) |z|^2+ \gamma^{(4)}_{2,2}(K)|z|^4 + h.o.t.
$$ 
Thus the twist matrix $\partial^2_{K,L} H(L,K)$ evaluated at $L=0$ is
$$
M(K):= \begin{pmatrix} \partial_{KK} \alpha^{(0)} & -\partial_{K}\alpha^{(2)} \\ -(\partial_{K}\alpha^{(2)})^T &  2\gamma^{(4)}_{2,2}
\end{pmatrix}.
$$
Note that $\e^{-1}M(K)$ is $\e$-independent.
We now compute the determinant of this matrix  at $K_\star$. Note that, since we are evaluating the matrix at $K_\star$, we do not need to compute explicitly the corrections to the order four given by the change of variables \eqref{change-ord3}, see formula \eqref{mizzi}. Direct computations show that $\det M(K_\star) \neq0$, thus the twist condition is fulfilled outside the zero set $\mathcal Z$ of an analytic function of $(E,K)\in\mathbb C^4$.
\end{proof}
\medskip

As we have mentioned before, $H(E,K)$ is real-analytic in $\mathcalligra D\;$ and satisfies the twist conditions outside $\mathcal Z$.  
Now we choose a point in $\mathcalligra D\;$ where we wish to prove persistence of tori by setting
\begin{equation}\label{parametri}
(E,K)=\xi+y=(\xi_0,\xi_1,\xi_2,\xi_3)+(y_0,y_1,y_2,y_3),
\end{equation}
where $\xi$ are parameters and 
with $y$ is in some  neighborhood $B_{r^2}$ of the origin  so that  the action $\xi+ r^2$ belongs to $\mathcalligra D\,$. We sistematically use the notation $$\mathtt F(\xi,\varphi)= \mathtt F(p(\xi_0,\xi_1\xi_2,\xi_3,\varphi),\xi_1,\xi_2,\xi_3,q(\xi_0,\xi_1\xi_2,\xi_3,\varphi)).$$

Now the Hamiltonian is
\begin{equation}\label{quasi}
\HH_{\mathtt {NLS}}=\muu (\xi)\cdot y+ \sum_{j=3,4}\mathcal H_j(\xi,\varphi,w_{\pm j})  +\sum_{j\neq \pm 1,\pm 2,\pm 3,\pm 4} (j^2+ \varepsilon f(\xi,\varphi) )|w_j|^2 + \mathcal R
\end{equation}
where 
$$
\muu(\xi) = \partial_\xi H (\xi) =: (0,1,1,4) + \varepsilon \laa(\xi)
$$
$$
\mathcal H_j= (j^2+ \varepsilon f(\xi,\varphi) )| (|w_j|^2+ |w_{-j}|^2)+ \varepsilon \mathcal U_j(\xi,\varphi){\rm Re}( w_j \bar w_{-j}) +\varepsilon \mathcal V_j(\xi,\varphi)|w_{-j}|^2 
$$
Note that we have grouped in $\mathcal R$ all the terms which are of degree higher than one in $y$ or greater that two in $w$. 
As before the change of variables is analytic in all its entries and we have  the bounds 
\begin{equation}\label{prisma}
|\cR|_{s,r,\cO}^1 \le {\mathtt C} (\e r + \frac{\e^2}{r^2})
\end{equation}
provided that we choose for $\cO$ a compact domain in $\mathcalligra D$ and choose $\e,r$ small enough, in particular much smaller that the distance between $\cO$ and the border of $ \mathcalligra D$.

The   term $\e^2 r^{-2}$  is due to  the  scaling of the domain in the $y$-component of the Hamiltonian vector field.

This change of variables gives for the constants of motion
$$
\mathbb L=  y_1 + y_2+y_3 +\sum_{j\neq \pm 1,\pm2}|w_j|^2\,,$$ 
$$\mathbb M= y_1-y_2-2y_3 +\sum_{j=3, 4}j(|w_j|^2+|w_{-j}|^2)+  \sum_{j\neq \pm 1,\pm2,\pm 3,\pm 4}j|w_j|^2
$$
We wish to reduce the ``normal form'' above to constant coefficients. First we remove the $\varphi$ dependence in $w_j$ for $j\neq \pm 1,\pm 2,\pm 3,\pm 4$.
This is 
done by performing the symplectic change of variables
\begin{equation}\label{phase2}
 y_0^{(\rm {new})}= y_0+ \frac{f(\xi,\varphi)-f_0(\xi) }{\laa_0(\xi)}  \sum_{j\neq \pm 1,\pm 2}|w_j|^2\,,\quad w_j^{(\rm {new})} =  w_j\exp\left(  \frac{-\ii \partial_\varphi^{-1}f(\xi,\varphi)}{\laa_0(\xi)}\right) \,,
\end{equation}
here $f_0(\xi)$ denotes the average of $f$ in $\varphi$ while $\partial_\varphi^{-1} f$ is the zero average primitive of  
$f- f_0 $.
Since $f$ is real this gives a phase shift (leaves each $|w_j|^2$ unchanged) and hence also the mass and momentum are unchanged.
 
Dropping the superscript ${}^{(\rm{ new})}$  the Hamiltonian is:
$$
\HH_{\mathtt {NLS}}=\muu (\xi)\cdot y+ \sum_{j=3,4}\widetilde{\mathcal H}_j(\xi,\varphi,w_{\pm j})  +\sum_{j\neq \pm 1,\pm 2,\pm 3,\pm 4} (j^2+ \varepsilon f_0(\xi) )|w_j|^2 + \mathcal R
$$
with
$$
\widetilde{\mathcal H}_j= (j^2+ \varepsilon f_0(\xi) ) (|w_j|^2+ |w_{-j}|^2)+ \varepsilon \mathcal U_j(\xi,\varphi){\rm Re}( w_j \bar w_{-j}) + \varepsilon \mathcal V_j(\xi,\varphi)|w_{-j}|^2, \quad j=3,4.
$$
Possibly restricting to a smaller set $\cO$, in order to assure that $\laa_0(\xi)$ is bounded away from zero
we have for $\cR$ the same bounds \eqref{prisma}.\\ 
Now the dependence on $\varphi$ appears only in the finite dimensional blocks given by $\widetilde{\mathcal H}_3,\widetilde{\mathcal H}_4$, where we reduce to constant coefficients by using Floquet's theorem. 
Consider the Hamiltonian
$$
H_{\mathtt{Fl}}=\e \laa_0 y_0 + \widetilde{\mathcal H}_3(\xi,\varphi,w_{\pm 3})+\widetilde{\mathcal H}_4(\xi,\varphi,w_{\pm 4})
$$
Direct computation shows that the Hamiltonian flow  preserves the quantities $|w_j|^2+|w_{-j}|^2$ for $j=3,4$. Since the reduction procedure is the same for $j=3,4$, we explicitly perform it on one block.

The Hamilton equations for $\mathtt w_j=(w_j, w_{-j})$ have the $2 \times 2$ block structure
\begin{equation}\label{ablocchi}
\begin{pmatrix}
\dot{\mathtt w_j} \\ \dot{\bar{\mathtt w}}_j
\end{pmatrix}
= \ii j^2 \begin{pmatrix}
\mathtt w_j \\ \bar{\mathtt w}_j
\end{pmatrix}+ \e
\begin{pmatrix}
A_j(\xi,\varphi) & 0 \\ 0 & \overline{A_j (\xi,\varphi)}
\end{pmatrix}
\begin{pmatrix}
\mathtt w_j \\ \bar{\mathtt w}_j
\end{pmatrix} ,
\end{equation}
where the $2\times 2$ matrix
\begin{equation}\label{AAA}
A_j(\xi,\varphi) := \ii  \begin{pmatrix}
f_0 & \mathcal U_j \\  \mathcal U_j &  f_0+\mathcal V_j
\end{pmatrix}
\end{equation}
is skew-Hermitian and $\overline{A _j(\xi,\varphi)}$ denotes the entry-wise complex conjugate of $A_j(\xi,\varphi)$. In other words, the two Lagrangian subspaces $\mathtt w_j=0$, $\bar{\mathtt w}_j=0$ are invariant, 
the system is decoupled into
\begin{equation}\label{fiocco1}
\dot \varphi= \e \laa_0\,,\quad \dot{\mathtt w}_j = (\ii j^2\mathbb 1+\e A_j(\xi,\varphi)) \mathtt w_j
\,,\quad \dot{\bar{\mathtt w}}_j =( -\ii j^2\mathbb 1 +\e \overline{A_j (\xi,\varphi)}) \bar{\mathtt w}_j\ ,
\end{equation}
By variation of constants we may set $\mathtt w_j= \mathtt z_j e^{\ii j^2 t}$ so that the equations for $\mathtt z_j$  are
\begin{equation}\label{fiocco1bis}
\dot \varphi= \e \laa_0\,,\quad \dot{\mathtt z}_j = \e A_j(\xi,\varphi) \mathtt z_j
\,,\quad \dot{\bar{\mathtt z}}_j =\e \overline{A_j (\xi,\varphi)}\bar{\mathtt z}_j\ ,
\end{equation}
which is still a Hamiltonian system with Hamiltonian 
$$
H_{\mathtt{Fl}}=\e\left( \laa_0 y_0 + \sum_{j=3,4}  f_0(\xi)  (|w_j|^2+ |w_{-j}|^2)+  \mathcal U_j(\xi,\varphi){\rm Re}( w_j \bar w_{-j}) + \varepsilon \mathcal V_j(\xi,\varphi)|w_{-j}|^2 \right).
$$
Let us study this system and rescale the time to $\tau= \e t$.
\begin{theorem}[Floquet's theorem]\label{Floquet.complex}
Consider a complex linear $n\times n$ dynamical system $\dot x= A(\tau ) x$ where the matrix $A$ is periodic of period $T$. Let $X(\tau )$ be the fundamental matrix solution. Then 
$$
X(\tau )= P(\tau ) e^{\tau  B}
$$ 
where the $n\times n$ matrices $P(\tau )$ and $B$ satisfy:
\begin{itemize}
\item $P(\tau )$ is invertible and periodic of period $T$.
\item $B$ is time independent and satisfies $X(T)= e^{T B}$.
\end{itemize}

\end{theorem}

As a consequence, the dynamical system $\dot x = A(\tau) x$ is symplectically conjugated to the constant coefficient system $\dot v = B v$ through the time-periodic change of variables $P(\tau)$.

For $j=3,4$, the system \eqref{fiocco1bis} satisfies all the assumptions of Theorem \ref{Floquet.complex} and therefore we may reduce it to constant coefficients. From the proof of Floquet's theorem, it also follows that we can perform such reduction in a way that preserves the skew-Hermiticity of the matrix. In fact, the matrix $B$ in Theorem \ref{Floquet.complex} is constructed considering the matrix fundamental solution of
\begin{equation}
\begin{cases}
\dot W_j = A_j(\tau) W_j\\
W_j(0)=\mathbb 1
\end{cases}
\end{equation}
 and finding a constant coefficient matrix $B_j$ such that $e^{TB_j}=W_j(T)$. The key point is that the $2\times 2$ skew-Hermitian matrices form the Lie algebra $\mathfrak{su}(2)$ associated to the special unitary group $SU(2)$, which is compact and connected. This implies that $W_j(t)\in SU(2)$ for all $t$ and, in particular, $W_j(T)\in SU(2)$. Hence, using the fact that the exponential map is surjective on compact connected Lie groups, the matrix $B_j$ can be chosen so that $B_j\in\mathfrak{su}(2)$. Thus, the time-periodic change of variables $P_j(t):=W_j(t)e^{-tB_j}$ conjugates \eqref{fiocco1} to the constant coefficient dynamical system (with an abuse of notation, we denote the new variables again with $\mathtt z_j$)
\begin{equation}\label{cc1}
\dot \varphi=\e \laa_0 \,,\quad \dot{\mathtt z}_j = \e B_j(\xi) \mathtt z_j \,,\qquad B_j(\xi)=\ii \begin{pmatrix} a_j & \bar c_j\\ c_j & b_j\end{pmatrix}\in\mathfrak{su}(2)\ ,
\end{equation}
note that we have returned to the original time-scale.

Clearly, the change of variables $\overline{P_j(t)}:= \overline{W_j(t)} e^{-t \overline B_j}$ conjugates the last equation in \eqref{fiocco1bis} to
\begin{equation}
\dot{\bar{\mathtt z}}_j= \e\overline{B_j(\xi)} \bar{\mathtt z}_j\ .
\end{equation}
We now return to our original system \eqref{fiocco1}. We note that the fundamental solution $W(\tau)$ of \eqref{fiocco1} is 
$W(\tau) = e^{\ii j^2\e^{-1}\tau} X(\tau)$ so if we apply the change of variables  given by $P$ to \eqref{fiocco1} we get 
\begin{equation}\label{cc1bis}
\dot \varphi=\e \laa_0 \,,\quad \dot{\mathtt w}_j = (\ii j^2 +\e B_j(\xi) )\mathtt w_j \,.
\end{equation}
Adding to $y_0$ a correction quadratic in the $w_{\pm j }$ with $j=3,4$, we obtain a symplectic change of variables.
Thus,  the Hamiltonian becomes
\begin{equation}\label{morte}
\HH_{\mathtt {NLS}}=\muu(\xi)\cdot y + \sum_{j=3,4}\mathcal Q_{j}(\xi,w_j,w_{-j},\bar w_j, \bar w_{-j})+ \sum_{j\neq\pm1,\pm2,\pm3,\pm4} (j^2 + \varepsilon f_0(\xi)) |w_j|^2 + \mathcal R
\end{equation}
with
$$
\mathcal Q_{j}= j^2(|w_j|^2+|w_{-j}|^2)+\e\big\{a_j|w_j|^2 + b_{j}|w_j|^2+ \text{Re}( c_j w_j \bar w_{-j})\big\}.
$$
The  mass and momentum become
$$
\mathbb L= y_1+y_2 + y_3 + \sum_{j\neq \pm 1,\pm2} |w_j|^2\,,$$
$$
\mathbb M =y_1-y_2-2 y_3 +\sum_{j=3,4} j (|w_j|^2+ |w_{-j}|^2)+  \sum_{j\neq \pm 1,\pm2,\pm 3,\pm 4} j |w_j|^2
$$
since  the Floquet change of variables must preserve $ |w_j|^2+ |w_{-j}|^2$ for $j=3,4$.
By the classification of quadratic Hamiltonians, since $B_j(\xi)$ is skew-hermitian   there exists a  symplectic change of variables, depending on $\xi\in \cO$ smoothly,  where the $\mathcal Q_j$ are diagonal, i.e. 
 \begin{equation}\label{parto}
\HH_{\mathtt{NLS}}= \cN+ \mathcal R
\quad \cN= \muu(\xi)\cdot y +  \sum_{j\neq\pm1,\pm2,} \Omega_j |w_j|^2 
 \,,\quad \Omega_j= (j^2 + \varepsilon \Theta_j(\xi)) \,,
 \end{equation}
Where we have denoted by $\Theta_{\pm j}$ the eigenvalues of $-\ii  B_j$ for $j=3,4$ and for $|j|>4$ we have set $\Theta_j= f_0$.
Finally $\cR$ satisfies bounds of the form \eqref{prisma}.

 \subsection{Melnikov conditions}\label{melni}
 
 As in the paper \cite{PP2}, we study the action of the NLS normal form
 $$
 \cN :=  \muu(\xi)\cdot y +  \sum_{j\neq\pm1,\pm2,} \Omega_j(\xi) |w_j|^2
 $$
 by Poisson bracket on the space of regular analytic Hamiltonians  that are at most quadratic in the normal modes and commute with  mass and momentum.
 
 We only need to study the action of $\cN$ on monomials of the form 
 $$ e^{\ii\ell\cdot\psi} w^{\sigma'}_h w_k^\sigma$$ 
 where $\ell=(\ell_0,\ell_1,\ell_2,\ell_3)$, $\sigma\in\{0,1,-1\}$ and by definition
 $$
 w_k^\sigma:=
 \begin{cases}
 w_k \quad \text{if }\sigma=1\\
 \bar w_k \quad \text{if }\sigma=-1\\
 1 \quad \text{if }\sigma=0
 \end{cases}
 .
 $$
Mass and momentum conservation means that we  only need to consider monomials satisfying the following constraints:
\begin{equation}\label{mass.conservation}
\eta(\ell)+\s'+\sigma = 0
\end{equation}
\begin{equation}\label{momentum.conservation}
\pi(\ell)+\s'\fr(h)+\sigma\fr(k) = 0
\end{equation}
where
$$
\eta(\ell):= \ell_1 +\ell_2 +\ell_3
\,,\quad 
\pi(\ell):= \ell_1 -\ell_2 -2\ell_3
\,,\quad 
\fr(j):=
 \begin{cases}
 |j| \quad \text{if } j=3,4\\
 j \quad \text{if }|j|\geq5
 \end{cases}
$$

\begin{proposition}\label{melnikov}
For each $\ell,\s,\s',h,k$ satisfying \eqref{mass.conservation} and \eqref{momentum.conservation} the Melnikov resonance condition 
\begin{equation}\label{melnikov.res}
\muu\cdot \ell +\s \Omega_h+\s' \Omega_k=0
\end{equation}
defines a proper algebraic surface except in the trivial case $\ell=0$, $\s+\s'=0$, $h=k$.
\end{proposition}
\begin{proof}
It is sufficient to verify that none of these analytic functions are identically zero.

We first consider the second order Melnikov conditions, namely $\s, \s' \neq 0$.
Since these are all  affine functions of $\e$, these functions are identically zero if and only if
\begin{equation}\label{resonant.monomial}
\pi^{(2)}(\ell) + \s\fr^2(h)+\sigma'\fr^2(k) = 0 \,,\quad \laa\cdot \ell  + \sigma \Theta_h +\sigma' \Theta_k\equiv 0
\end{equation}
where 
$$
\pi^{(2)}(\ell) := \ell_1 + \ell_2 +4 \ell_3.
$$
In conclusion we have the system of equations
\begin{empheq}[left=\empheqlbrace]{align}
&\ell_1+\ell_2+\ell_3+\s +\s'=0  \label{pippa1} \\
&\ell_1-\ell_2-2\ell_3 +\s \fr(h)+\s' \fr (k)=0 \label{pippa2} \\ 
&\ell_1+\ell_2+4\ell_3 + \s\fr^2(h)+\sigma'\fr^2(k) = 0 \label{pippa3} \\
& \laa\cdot \ell  + \sigma \Theta_h +\sigma' \Theta_k=0 \label{pippa4}
\end{empheq}
\textit{Case 1}: $|h|,|k|>4$.\\
If $\s+\s'=0 $, in the last equation we get $\laa\cdot \ell=0$. Since $\xi\to\laa$ is generically a diffeomorphism this gives $\ell=0$, hence $h=k$ from \eqref{pippa2}.
If $\s+\s'=2$ we get $\laa\cdot \ell +2 f_0(\xi)=0$. In order to prove that this is not an identity we only need to compute the functions $\laa(\xi)$, $f_0(\xi)$ as the action $\xi_0=E \to 0$ and $(\xi_1,\xi_2,\xi_3)=K\to (4,0,2)$. It is easily seen that
$$\laa(0,4,0,2)= (\alpha^{(2)}(4,0,2),\partial_K \alpha^{(0)}(4,0,2))= (-144\sqrt{3},426,426,498)\,,$$
$$
\quad f_0(0,4,0,2)= f(1,4,0,2) =558$$
By the irrationality of $\lambda_0(0,4,0,2)$, we deduce $\ell_0=0$. Then we try to solve
\begin{equation}\label{miracle}
\laa(\xi)\cdot \ell +2 f_0(\xi)=0 \,,\quad \ell_1+\ell_2+\ell_3=-2
\end{equation}
at $\xi=(0,4,0,2)$ and one verifies that no integer solutions $\ell$ exist, since (set $\ell_1 + \ell_2 = x$, $\ell_3=y$) the solution of the linear system
$$
\begin{cases}
426x+498y=-1116\\
x+y=-2
\end{cases}
$$
is $x=5/3$, $y=-11/3$.
\\
\textit{Case 2}: $h,k\in\{\pm3,\pm 4\}$.\\
As above, we consider the limit $\xi_0=E \to 0$ and $(\xi_1,\xi_2,\xi_3)=K\to (4,0,2)$. We compute $\mathcal V_j (0,4,0,2)=0$ for $j=3,4$, and $\mathcal U_3(0,4,0,2)=144$, $\mathcal U_4(0,4,0,2)=18$. This gives $\Theta_{\pm 3} (0,4,0,2)= 558 \pm 144$, $\Theta_{\pm 4} (0,4,0,2)= 558 \pm 18$. Then, since $\laa_1, \laa_2, \laa_3, \Theta_h, \Theta_k$ are all integers while $\laa_0$ is irrational (here all functions, when not specified, are evaluated at $\xi=(0,4,0,2)$), equation \eqref{pippa4} gives $\ell_0=0$. First consider the case $\sigma + \sigma' =0$. Then equations \eqref{pippa1}, \eqref{pippa3} imply $3 \ell_3 + \s\fr^2(h)+\sigma'\fr^2(k) = 0$. Since $3$ does not divide $4^2-3^2=7$, we get $|h|=|k|\in\{3,4\}$, hence $\fr(h)=\fr(k)$. Equations \eqref{pippa1}, \eqref{pippa2}, \eqref{pippa3} then give $\ell=0$. The observation that $\Theta_j\neq\Theta_{-j}$ for $j=3,4$ implies that \eqref{pippa4} with $\ell=0$ is satisfied only if $h=k$. Now consider the case $\sigma+\sigma'=2$. If $|h|=|k|$, then equations \eqref{pippa1}, \eqref{pippa4} lead to the linear system
$$
\begin{cases}
426x+498y=- 1116 +  \rho\\
x+y=-2
\end{cases}
$$
with $\rho\in\{0, \pm 36, \pm 288\}$, which has no integer solution; indeed the first equation (divided by $3$) implies $x+y \equiv 0 \pmod{3}$, which is clearly incompatible with the second equation. If $|h|\neq|k|$, then equations \eqref{pippa1}, \eqref{pippa3} imply $3\ell_3 = -23$, so there is no integer solution.
\\
\textit{Case 3}: $|h|\in\{3,4\}$, $|k|>4$.\\
Like before, we evaluate everything at $\xi=(0,4,0,2)$; again, we use equation \eqref{pippa4} to deduce $\ell_0=0$. Consider the case $\sigma + \sigma' = 2$. Then equations \eqref{pippa1}, \eqref{pippa4} lead to the linear system
$$
\begin{cases}
426x+498y=- 1116 +  \varrho\\
x+y=-2
\end{cases}
$$
with $\varrho\in\Xi:=\{\pm 18, \pm 144\}$, which has no integer solution as above. Finally, if $\sigma + \sigma' = 0$, equations \eqref{pippa1}, \eqref{pippa4} imply $72\ell_3 = \varrho$, with $\varrho\in\Xi$. Now, if $\varrho = \pm 18$, there is no integer solution, which rules out the case $|h|=4$. Therefore we have $|h|=3$, $\ell_3=\pm 2$. Using equations \eqref{pippa1}, \eqref{pippa3}, this implies $\fr^2(k) \in \{\pm3, \pm 15\}$, which is not possible.

\bigskip

Then we consider first order Melnikov resonances of the type, without loss of generality $\s=-1,\s'=0$.
$$
\muu\cdot \ell - \Omega_h
$$
cannot identically vanish. Indeed, it is enough to consider the mass conservation $\ell_1+\ell_2+\ell_3=1$ and the resonance condition $\laa\cdot\ell = \Theta_h$. Reasoning like above, the evaluation at $\xi=(0,4,0,2)$ leads to the linear system
$$
\begin{cases}
426x+498y= 558 +  \varrho\\
x+y= 1
\end{cases}
$$
with $\varrho\in\Xi$, which again has no solution, since the first equation implies $x+y\equiv0\pmod3$.

\bigskip

The case $\s=\s'=0$ is completely trivial since the map $\omega \leftrightarrow \xi$ is a local diffeomorphism. This concludes the proof of the proposition.
\end{proof}
Consider the  Hamiltonian \eqref{parto} in the set $\cO$. Let
$M_0,L_0$ be defined by 
\begin{equation}\label{inizio}
|\laa |_\infty+ |\laa |^{\rm lip}, |f_0|_\infty + |f_0|^{\rm lip} ,\sup_{j=\pm3,\pm 4} |\Theta _j|_\infty+ | \Theta _j|^{\rm lip}\le M_0 \,,\quad |(\laa )^{-1}|^{\rm lip}\le L_0 
\end{equation}
where given a map  $f:\cO\to \R^d$ we set
$$
|f|_\infty:= \sup_{\xi\in\cO}\sup_{i=1,\ldots, d}|f_i|\,,\quad |f|^{\rm lip}:=  \sup_{\xi\neq \eta\in\cO}\sup_{i=1,\ldots, d}\frac{|f_i(\xi)-f_i(\eta)|}{|\xi-\eta|}.
$$
\begin{lemma}\label{melquant}
There exists an $\e$-independent  proper algebraic surface $\mathfrak A$  so that for all $\xi\in \cO\setminus \mathfrak A$ on has 
$$
\omega\cdot \ell +\s \Omega_h+\s' \Omega_k\neq 0 
$$
for all non-trivial $\ell,\s,\s',h,k$ satisfying  $|\ell|\le 4 M_0 L_0$ and conditions \eqref{mass.conservation},\eqref{momentum.conservation}.
Let  $\cO_0$ be an $\e$-independent  compact domain in $\cO\setminus \mathfrak A$. There exist  constants $\al_0$, $\mathtt R_0$ such that  
\begin{eqnarray} \label{stimm0}&
 |\laa \cdot \ell +\s \Theta _h +\sigma' \Theta _k|\ge \alpha_0 \quad \forall \sigma,\sigma'= 0,\pm 1, \;|\ell|< 4 M_0L_0\,, \mbox{\eqref{mass.conservation},\eqref{momentum.conservation} hold}\nonumber
\\ &
 |\cR|_{s,r,\OO_0}^{\gamma_0} \le  \gamma \e \mathtt R_0 \,,\quad \gamma_0= \gamma M_0^{-1}\,,
\end{eqnarray}
where $\cR$ is defined in \eqref{parto}, $\al_0$ does not depend on $\e$ while $\mathtt R_0\sim r +\e r^{-2}$.
\end{lemma}
\begin{proof}
By Proposition \ref{melnikov} each Melnikov resonance defines a proper algebraic surface. Our first statement follows by showing that the condition $|\ell|\le 4 M_0L_0$ implies: 
1. that the only  Melnikov resonance surfaces  which may intersect $\cO$ are such that
\begin{equation}\label{pipi}
\pi^{(2)}(\ell) + \s\fr^2(h)+\sigma'\fr^2(k) = 0 
\end{equation} see \eqref{resonant.monomial};
 2. that such resonances are only a finite number. This is due to the smallness of $\e$.   Indeed if \eqref{pipi} does not hold, we have
 $$ |\laa \cdot \ell + \s \Theta _i +\sigma' \Theta _j|\ge  |\pi^{(2)}(\ell) + \s\fr^2(h)+\sigma'\fr^2(k)| - \e (4M_0^2L_0 + 2 M_0) \ge 1-\frac12.
 $$
Then we notice that \eqref{pipi} fixes $h,k$ inside a ball of radius $4M_0L_0$ except in the trivial case $\ell=0,\s=-\sigma', h=k$.
The estmates \eqref{stimm0} follow trivially.
\end{proof}
\subsection{Proof of Proposition \ref{formanorm}}\label{prova1}
We  just apply all the changes of variables discussed in the previous sections.First we apply  the Birkhoff change of variables generated by the Hamiltonian in \eqref{birk1}, then we pass the sites $(\pm 1,\pm 2)$ to polar coordinates  in \eqref{act1} and pass to the coordinates \eqref{act2}. Next we apply our first phase shift \eqref{phase1} and then we integrate $\mathcalligra H\;$  in \eqref{integro}. Now the  change of variables \eqref{parametri}  sets our approximately invariant torus at $y=0$, $z=0$. The NLS Hamiltonian now has the form \eqref{quasi}, and the normal form depends only on one angle. In order to remove this dependence we apply the phase shift \eqref{phase2}  and then Floquet's theorem \ref{Floquet.complex}. We obtain the Hamiltonian \eqref{morte}, which we diagonalize by using the standard theory of quadratic Hamiltonians. We choose $\cO_0$ to be a compact domain as in Lemma \ref{melquant}. Finally, formula \eqref{battimenti davvero} is obtained as a consequence of \eqref{quantobatte}, by possibly further restricting the set $\cO_0$.
\qed
\section{KAM theorem}\label{sec:KAM}
\subsection{Technical set-up} 
We introduce a degree decomposition on  $\cA_{s,r}$  by associating to each monomial $\mathfrak  m$ a degree $d(\mathfrak m)$ as follows
$$\mathfrak m_{\ell,j,\a,\bt}= e^{i \ell\cdot \varphi}y^j z^\al \bar z^\bt\quad  \to \quad d(\mathfrak m_{\ell,j,\a,\bt})= 2j+|\al|+|\bt|-2.$$ 

 This gives us a definition of homogeneous polynomials of degree $d$ and a degree decomposition of analytic functions. Given  $f\in \cA_{s,r}$ we denote by $f^{(d)}$ its projection onto the homogeneous polynomials of degree $d$
 $$
 f= \sum f_{\ell,j,\a,\bt}e^{i \ell\cdot \varphi}y^j z^\al \bar z^\bt \,,\quad \Pi^{d}f\equiv f^{(d)}= \sum_{2j+|\al|+|\bt|-2=d} f_{\ell,j,\a,\bt}e^{i \ell\cdot \varphi}y^j z^\al \bar z^\bt\,,
 $$
We use the same notations for $\Pi^{\le d}$ and $\Pi^{\ge d}$. We are also interested in projections onto trigonometric polynomials in $\varphi$, we shall denote
$$ \Pi_N f:=  \sum_{|\ell|\le N} f_{\ell,j,\a,\bt}e^{i \ell\cdot \varphi}y^j z^\al \bar z^\bt \,,\quad
\Pi_N^\perp:= \mathbb 1 - \Pi_N $$
More in general given a subset of indices $\mathbb I\subseteq \Z^4\times\N^4\times \N^\Z\times \N^\Z$  we define
\begin{equation}\label{momo}
\Pi_{\mathbb I}f= \sum_{(\ell,j,\al,\bt)\in \mathbb I} f_{\ell,j,\a,\bt}e^{i \ell\cdot \varphi}y^j z^\al \bar z^\bt.
\end{equation}
  We denote as usual by $\{A,B\}$ the associated Poisson bracket and, if we want to stress the role of one of the two variables, we also write $ad(A)$ for the linear operator $B\mapsto \{A,B\}$.\footnote{$ad$ stands for {\em adjoint} in the language of Lie algebras.}
  The main properties of the majorant norm \eqref{weno}, contained in  \cite{BBP1}, Lemmata  2.10,\ 2.15,\  2.17, express the compatibility of the norm with projections and Poisson brackets, we briefly recall them in the next propositions.
 \begin{proposition}\label{riassunto}
 For every $r, s>0$ the following holds true:
 \begin{itemize}
 \item[(i)] All projections  $\Pi_\mathbb I$ are continuous namely
 $$
  |\Pi_\mathbb I h|^\gamma_{s,r}\leq |h|^\gamma_{s,r}
 $$
  {\em Smoothing}: one has
 $$
 |\Pi_{\leq N} h|^\gamma_{s,r}\leq N^{{s_1}} | h|^\gamma_{s-s_1, r}\,,\quad |\Pi_{\geq N} h|^\gamma_{s,r}\leq N^{-s_1} | h|^\gamma_{s+s_1,r}
 $$
 \item [(ii)]{\em Graded Poisson algebra}:
 Given $f,g\in \cA_{s,r}$ for any $r'<r$ one has
 $$
 |\{f,g\}|^\gamma_{s, r'}\leq (1-\frac{r'}{r})^{-1} C(s) |f|^\gamma_{s,r}|g|^\gamma_{s,r}
 $$
 moreover on all monomials $d(\{f,g\})= d(f)+d(g)$.
 \item[(iii)]{\em partial ordering} if we have
 $$
 |f_{\al,\bt,\ell}|\leq |h_{\al,\bt,\ell}| \,,\quad \forall \; \al,\bt,\ell
 $$
 and $h_{\al,\bt,\ell}$ is the Taylor-Fourier expansion of a function $h\in \cA_{s,r}$ then there exists a unique function $f$ whose Fourier expansion is
 $\{f_{\al,\bt,\ell}\}$ and such that
 $$ |f|^\gamma_{s,r}\leq |h|^\gamma_{s,r}$$
 \item[(iv)] {\em Degree decomposition}
 Given a  Hamiltonian $h\in \cA_{s,r}$ which is homogeneous of degree $d$ then  $h\in \cA_{s,R}$ for all $R$ and  one has 
 $$
 |h|^\gamma_{s,r} \leq r^d |h|^\gamma_{s,1}
 $$ 
 \end{itemize}
 \end{proposition}
 \begin{proposition}\label{riassunto2}
 For every $r, s>0$ the following holds true:
 \begin{itemize}
 \item[(i)] {\em Changes of variables}:
 if $(1-\frac{r'}{r})^{-1} |f|^\gamma_{s,r}<\varrho$ sufficiently small then its Hamiltonian vector field $X_f$ defines a close to identity canonical change of variables
 $\cT_f$ such that:
 $$
 h\circ\cT_f= e^{\{\cdot,\cdot\}} h \quad{\rm satisfies}\quad |h\circ\cT_f|^\gamma_{s, r'} \leq (1+C\varrho)|h|^\gamma_{s,r}
 $$
 \item[(ii)]{\em Remainder estimates}: consider  two Hamiltonians $f,g$ with $f$  of minimal  degree $\td_f$ and $g$ of minimal degree  $\td_g$, then set
 \begin{equation}\label{taylor}
 \re{\ti}(f,h)=\sum_{l=\ti}^\infty  \frac{(-\ad f)^l}{l!} h\,,\quad \ad(f)h:= \{h,f\}
 \end{equation} 
 then  $\cP\!_{\ti}(f,g)$ is of minimal degree $\td_f \ti   +\td_g$ 
 and we have the bounds
 $$
 \abs{\re{\ti}(f,h)}^\gamma_{s,r'} \leq C(s) \left(1-\frac{r'}{r}\right)^{-\ti} (|f|_{s,r}^\gamma)^\ti |g|^\gamma_{s,r}.
 $$
 Note that the same holds if we substitute in \eqref{taylor}  the sequence $\{\frac{1}{l!}\}$  with any sequence $\{b_l\}$ such that  $\forall l$ one has  $|b_l|\leq \frac{1}{l!}$.
 \end{itemize}
 \end{proposition}
Note that  $\mathcal H$ defined in \eqref{NLS1} is $M$-regular\footnote{it is well known that the NLS is  locally well posed under much weaker regularity conditions. This is not the purpose of the present paper.} for all $r$.
\subsection{The iterative  algorithm} 
%
%
%
%

\begin{definition}
We first define the diagonal Hamiltonians which are our {\em normal forms}. 
$$
\mathcal F_{\rm ker}:= \{h= q\cdot y + \sum_{j\neq \pm1,\pm 2} Q_j|z_j|^2\,,\quad \mbox{with}\; q\in \R^4\,,\;Q_j\in \R\} \,, 
$$
Given an $M$-analytic Hamiltonian $H$ we denote by $\Pi_{\rm ker} H$ the projection of $H$ onto $\mathcal F_{\rm ker}$ (this is a projection of the form \eqref{momo}).
Now we define
$
\cF_{s,r}
$ be the subspace of $M$-analytic Hamiltonians $H$  such that $H-\Pi_{\rm ker} H\in \cA_{s,r}$ is a regular M-analytic Hamiltonian and Poisson commutes with $\mathbb L,\mathbb M$. Finally let 
$$\cF_{\rm rg}= (\mathbb 1-\Pi_{\rm ker})\Pi^{\le 0} \cF_{s,r}.$$
so that
\begin{equation}
\cF_{s,r} = \mathcal F_{{\rm ker}}\oplus  \mathcal F_{{\rm rg}}\oplus \mathcal F_{s,r}^{>0}.
\end{equation}
Note that $\cF_{s,r}$ is naturally decomposed in terms of subspaces of growing degree.
\end{definition}
By construction the NLS    Hamiltonian
 \begin{equation}
\HH_{\mathtt{NLS}}= \omega(\xi)\cdot y +  \sum_{j\neq\pm1,\pm2,} \Omega_j(\xi) |w_j|^2 + \mathcal R= \mathcal N +\cR \in \cF_{s,r} \,,\quad \forall r<r_0,s<s_0
 \end{equation}
\begin{remark}[The goal]\label{theg}
By definition, the normal form $\mathcal N$  is in $\mathcal F_{{\rm ker}}$.
In general, the condition for  a Hamiltonian $  H=  N+  P,\   N=\Pi_{{\rm ker}}   H$ to have  KAM tori is $\Pi_{{\rm rg}}   H=\Pi_{{\rm rg}}   P=0$. So our goal is to find a
  symplectic transformation $  \Psii_\infty$ so that $$ \Pi_{{\rm rg}}      (  H)=0. $$\end{remark}
The strategy is to construct this as a limit of   a {\em quadratic Nash-Moser algorithm}.  

We will  denote:  
\begin{equation}\label{notazioni}
  N:=\Pi_{{\rm ker}}(  H)\,,\quad 
  P_{{\rm rg}}:=\Pi_{{\rm rg}}(  H)\,,\quad   P^{>0}:=\Pi^{>0}(  H)\,,
\end{equation}
so that $ H=   N+   P_{{\rm rg}} +  P^{>0}$.
 \smallskip

By the Poisson algebra property (ii) of Proposition \ref{riassunto},
if  $A$ has degree $i>0$  then also $ad(A)$  has  positive degree and hence is strictly lower triangular on  $\mathcal F_{s,r}$ w.r.t.  the degree decomposition.
 
\bigskip 
 
We start  with  the NLS Hamiltonian \eqref{parto} which we denote by $  H_{0}:=   N_0 +   P_{{{\rm rg}}, 0} +   P^{>0}_0$, where $  N_0=\mathcal N$,  $ P_{{{\rm rg}}, 0} +   P^{>0}_0=\cR $, so that  $$   P_{{\rm rg}, 0}\sim \e^2\,,\quad P^{>0}_0\sim \e r$$  are appropriately small. 

We wish to find a convergent sequence of changes of variables
\begin{equation}\label{meroda}
\Psii_{m+1}:= e^{ad (F_{m })} \circ   \Psii_{m },
\end{equation}   dependent on a sequence $K_m$ of ultraviolet cuts, so that at each step $  H_{m+1}=   \Psii_{m+1}(  H_{0})= N_{m+1} + P_{{\rm rg}, m+1}^{\le 0} + P_{m+1}^{>0}$ 
is such that $   N_{m }$ stays close to $\mathcal N$,  $  P^{>0}_m$ stays bounded while $   P_{{\rm rg}, m}$ converges to zero (super--exponentially). 

At a purely formal level we would like that  $   P_{{\rm rg}, m+1}$ is 
{\em quadratic} w.r.t. $   P_{{\rm rg}, m }$.
The generating function $F_{m}\in \mathcal F_{{\rm rg},\leq K_{m+1}}:= \Pi_{\leq K_{m+1}}\mathcal F_{\rm rg} $ is fixed by solving the homological equation
 \begin{equation}\label{hoeq}
\{   N_{m},F_{m}\}+ \Pi_{{\rm rg},m} \{P^{>0}_{m},F_{m}\}=\Pi_{\leq K_{m+1}} P_{{\rm rg},m}\,,\quad  \Pi_{{{\rm rg}},m}:= \Pi_{{{\rm rg}},\leq K_{m+1}}
\end{equation}
which uniquely determines $F_{m}$ as a {\em linear} function of $  P_{{\rm rg},m}$ provided that the linear operator:
$$\mathfrak L_{m}:= {\rm ad}(  N_{m}) + \Pi_{{{\rm rg}},m} {\rm ad}(P^{>0}_{m})= {\rm ad}(  N_{m}) + \Pi_{{{\rm rg}},m} {\rm ad}(P^{1}_{m})+ \Pi_{{\rm rg},m} {\rm ad}(P^{2}_{m}) , $$ is invertible on $\mathcal F_{{\rm rg},\leq K_{m+1}}$ (clearly we also need some quantitative control on the inverse).   
\begin{remark}\label{nilp}
On $\mathcal F_{s,r}$ the operators 
$ {\rm ad}(  N_{m}) , \Pi_{{\rm rg},m} {\rm ad}(P^{1}_{m}), \Pi_{{\rm rg},m} {\rm ad}(P^{2}_{m})$ have respectively degree 0,1,2  so it should be be clear that $\mathfrak L_{m}$ is invertible if and only if ${\rm ad}(  N_{m})$ is invertible and in this case  one inverts $$\mathfrak L_{m}=  {\rm ad}(  N_{m}) \bigg(1+{\rm ad}(  N_{m}) ^{-1} \Pi_{{\rm rg},m} {\rm ad}(P^{>0}_{m})\bigg)$$ by inverting the second factor. This is  of the form $1+A$ with $A$ a sum of two linear operators of degree 1,2 respectively, so $A^3=0$ and we invert $1+A$ with the 3 term Neumann series $1-A+A^2$.
\end{remark} 
We now justify our choice  by computing one {\em KAM step}, for notational convenience we drop the pedex $m$ in $  H_{m }$ etc..  and substitute $  H_{m+1}$ with $  H_+$ etc...\,.

Let us compute $  H_+:= e^{ad (F)}  H$.  
First split the operator  $e^{ad(F)}=1+ad(F)+E_F$, by definition $E_F$ is quadratic in $F$ and hence quadratic in $  P_{{\rm rg}}$. Regarding the term
$$(1+ad(F))(  N+   P_{{\rm rg}}+  P^{> 0})=   N+   P_{{\rm rg}}+  P^{> 0}-\{  N+  P^{> 0},F\} +\{F,  P_{{\rm rg}}\} $$
we first notice that, since $F$ is linear w.r.t. $  P_{{\rm rg}}$  then the last summand is quadratic moreover since $F$ solves the homological equation we have
$$
  P_{{\rm rg}}-\{  N+  P^{> 0},F\}= (\Pi_{{\rm ker}}+\Pi_{>0}+ \Pi_{>K}\Pi_{{\rm rg}})\{  P^{> 0},F\}+\Pi_{>K} P_{{\rm rg}}.
$$
Then we deduce that
$$
\Pi_{{\rm ker}} e^{ad( F)}  H:=  N_{+}=   N+\Pi_{{\rm ker}}\{  P^{> 0},F\}+\Pi_{{\rm ker}}Q(  P_{{\rm rg}})\,,$$
  $$ \Pi_{{\rm rg}} e^{ad (F)}  H:=  P^{\leq 0}_{+}= \Pi_{>K} (P_{{\rm rg}}+ \Pi_{{\rm rg}}\{  P^{> 0},F\} )+\Pi_{{\rm rg}}Q(  P_{{\rm rg}})\,,
$$
$$
 \Pi_{>0} e^{ad( F)}  H:=   P^{>0}_+=   P^{> 0}+\Pi_{>0}\{  P^{> 0},F\}+\Pi_{>0}Q(  P_{{\rm rg}})
$$
where $Q(  P_{\rm rg})\,,$  is quadratic in $  P_{{\rm rg}}$ and collects the terms from $E_F(  H)$ and $\{F,  P_{{\rm rg}}\} $.

We now introduce some parameters which control $H_m$. Recall that by definition of $\cF_{\rm ker}$
$$
N_m= \muu^{(m)}\cdot y + \sum_{j\neq \pm 1,\pm 2} \Omega_j^{(m)}|w_j|^2,
$$
and set
$$
\muu^{(m)}= (0,1,1,4) + \e \laa^{(m)}\,,\quad \Omega_j^{(m)}= j^2+\e \Theta_j^{(m)}.
$$
 Fix a small $\gamma>0$. Let $\OO_m$ be a positive measure Cantor set, assume that $\xi\to \laa^{(m)}(\xi)$ is invertible and let $L_m,M_m,\mathtt R_m,\alpha_m\geq 0$  be  such that
\begin{eqnarray} \label{stimm}&
 |\laa^{(m)}|+ |\laa^{(m)}|^{\rm lip}\le M_m\,,\quad  |\Theta^{(m)}_j|+ | \Theta^{(m)}_j|^{\rm lip}\le M_m \,,\quad |(\laa^{(m)})^{-1}|^{\rm lip}\le L_m \nonumber
 \\
 &|\laa^{(m)}\cdot \ell + \Theta^{(m)}_i +\sigma \Theta^{(m)}_j|\ge \alpha_m \quad \forall \sigma= 0,\pm 1, \;\forall (\ell,i,j)\neq (0,i,i): |\ell|< 4 M_0L_0\nonumber
\\ &
 |H_m-N_m|_{s,r,\OO_m}^{\gamma_m} \le  \gamma \e \mathtt R_m \,,\quad \gamma_m= \gamma M_m^{-1}.
\end{eqnarray}
Note that our conditions are fulfilled by the Hamiltonian \eqref{parto} for $m=0$.
 \begin{definition}\label{telpa}
We say that a positive parameter $b=\{b_m\}_{m\in \N}$ is telescopic if for each step $m$ we have $b_0/2< b_m< \frac32 b_0$ (usually $b_m$ is either an increasing or a decreasing sequence). 
\end{definition}
Note that a sufficient condition is that
\begin{equation}\label{tele}
\sum_{m=0}^\infty |b_{m+1}-b_m|\leq \frac{b_0}{2}.
\end{equation}

We choose as in \cite{PP16}:
\begin{equation}\label{srK}
 r_{m+1}= (1-2^{-m-3})r_m\,,\quad s_{m+1}= (1-2^{-m-3})s_m\,,\quad K_m=4^m K_0 .
\end{equation}
We choose the set $\mathcal O_{m+1}$ as
\begin{align}\label{meln}
\!\!\!\mathcal O_{m+1}:=&\{\xi\in \OO_m: \; |\muu^{(m)}\cdot \ell +\s \Omega^{(m)}_i+\s' \Omega^{(m)}_j|\geq {\e\gamma}K_m^{-\tau} ,\; \s,\s'=0,\pm1 ,\;  |\ell|\le K_m, \\
& \quad \eta(\ell)+\sigma+\sigma'=0\,, \quad \pi(\ell) + \sigma\mathtt r(i) + \sigma'\mathtt r(j)=0\,, \quad (\sigma\sigma',\ell,i)\neq(-1,0,j) \}. \nonumber
\end{align}
\begin{lemma}\label{friotta}
For all    $\xi\in\mathcal O_{m+1}$, \eqref{hoeq} admits a unique solution and moreover 
\begin{equation}\label{stF}
\|F_{m}\|^{\gamma_m}_{s'_m,r'_m,\OO_{m+1}} \leq (\e\gamma)^{-1}(\mathtt R_m K_m^{2\tau})^3\|P_{{\rm rg},m}\|^{\gamma_m}_{s_m,r_m,\OO_{m}}\,,\end{equation}  
where 
$s'_m= (s_m+s_{m+1})/2$,$r'_m= (r_m+r_{m+1})/2$ , and  
 $\gamma_m= \gamma M_m^{-1}$.
\end{lemma}
\begin{proof}
This is a subcase of  Lemma 14 of \cite{PP16} or Proposition 2.31 of \cite{CFP}. 
We first notice that \eqref{meln} and item (iii) of Proposition \ref{riassunto} imply that for any $g\in \cF_{{\rm rg},m}$ one has
$$
|\mathfrak D^{-1} g |^{\gamma_m}_{s,r,\OO_{m+1}}\le 3(\e \gamma)^{-1}  K_m^{2\tau}  |g |^{\gamma_m}_{s,r,\OO_{m+1}}.
$$
Following   Remark \ref{nilp} we set $A= \mathfrak D^{-1}\Pi_{{\rm rg},m} {\rm ad}(P_m^{>0})$.  Using \eqref{stimm} and Propositions \ref{riassunto},\ref{riassunto2} we get  
$$| A^j g|_{s',r',\OO_{m+1}}^{\gamma_m} \le (K_m^{2\tau}\mathtt R_m )^j   |g |^{\gamma_m}_{s,r,\OO_{m+1}}\,,\quad j=1,2. $$
The thesis follows by Remark \ref{nilp}.
\end{proof}

\begin{proposition}\label{kamforma}
For $K_0,\tau$ large and for all $\e,r$ such that $r +\e r^{-2} \lessdot \gamma K_0^{-6\tau}$, at each step $m$, in the set $\mathcal O_m$, one has the estimate 
\begin{equation}\label{superex}
\|P_{{\rm rg},m}\|^{\gamma_m}_{s_m,r_m,\OO_{m}} \leq  (r +\e r^{-2}) e^{-\frac32^m}\,.
\end{equation}  Moreover  the constants $\mathtt R,L,M,\alpha$ are telescopic. Finally our algebraic algorithm converges on the set $\cap_m\mathcal O_m$, and we obtain a change of variables $\Psii_\infty:=\lim_{m\to \infty}\Psii_m$. The corresponding Hamiltonian $H_\infty$, has reducible KAM tori.
\end{proposition}
\begin{proof} These estimates are  all standard. The condition on $\e,r$ implies that  $P_0\sim \e r +  \e^2 r^{-2}$ is small w.r.t. the size of the divisor $\sim \e \gamma$.  
By induction, suppose we have reached some step $m$ and proved the estimates. Using \eqref{stF} and  items (i-ii) of Proposition \ref{riassunto2}, one proves  that  $F_{m }$ defines a symplectic change of variables with  $H_{m+1}:= e^{{\rm ad}F_m} H_m$  well defined in the domain $D(s_{m+1},r_{m+1})$.
The corrections of the parameters  $\mathtt R,L,M,\alpha$ are obtained by item (i) Proposition \ref{riassunto}  and item (ii) of Proposition \ref{riassunto2} using the fact that the coordinate change is very close to the identity due to the super--exponential decay of the norm of $F$. This implies easily the telescopic  nature of the parameters used,  as $|b_{m+1}-b_m|<$ const $b_0 e^{-(3/2)^m}$.  
%
Then  the algorithm converges.
\end{proof}

\subsection{Proof of Proposition \ref{KAM}.}
The change of variables of item (i) is $\Psii_\infty:=\lim_{m\to \infty} \Psii_m$ defined  in \eqref{meroda} and with $F_m$ defined in \eqref{hoeq}. The estimates on $\Psii_\infty-\mathbb 1$ follow from Lemma \ref{friotta}, Proposition \ref{kamforma} and Proposition \ref{riassunto2},  provided that we choose $r=\e^{1/3}$.\\
Item (ii) follows directly from Proposition \ref{kamforma} and Definition \ref{reducible}.

In order to complete the proof of the KAM theorem we only need to show that the set $\OO_\infty:=\cap_{m=0}^\infty\OO_m$ has positive measure.
Let us denote
$$
\mathfrak R_{\s,\s',\ell,i,j}^{(m)}:=\{\xi\in \OO_{m}:  |\muu^{(m)}\cdot \ell +\s \Omega^{(m)}_i+\s' \Omega^{(m)}_j|< {\e\gamma}K_m^{-\tau}\}
$$
so that
$$
\OO_{m+1}= \OO_m\setminus \cup_{\s,\s',\ell,i,j}^*\mathfrak R_{\s,\s',\ell,i,j}^{(m)}
$$
where the $\cup^*$ is the union of all the sets such that
$$
|\ell|<K_m\,,\quad \eta(\ell)+\sigma+\sigma'=0\,, \quad \pi(\ell) + \sigma\mathtt r(i) + \sigma'\mathtt r(j)=0\,, \quad (\sigma\sigma',\ell,i)\neq(-1,0,j).
$$

%
%
We first consider the case $\sigma\sigma'\neq0$, corresponding to \emph{second Melnikov conditions}.
We claim that for any fixed $\bar\ell\in\Z^4$, the union $\cup^*_{\s,\s',\bar\ell,i,j} $ is finite, with an upper bound on the cardinality depending on $|\bar\ell|$.

By the estimates \eqref{stimm} and the fact that $\alpha_m$ is telescopic we deduce that, if we choose $K_0$ so that $\alpha_0> 2 K_0^{-\tau}$, we have $\mathfrak R_{\s,\s',\ell,i,j}^{(m)}=\emptyset$ for all $|\ell|\le 6 L_0M_0$.

Denote $\mathtt v:= (0,1,1,4)$. If $|\s i^2+\s' j^2|> 2|\mathtt v| |\ell| $   we have $\mathfrak R_{\s,\s',\ell,i,j}^{(m)}=\emptyset$. Since $$|\s i^2+\s' j^2|> \big||i|-|j|\big|(|i|+|j|),$$ this condition ensures  that $\cup^*$ runs only over $$|i|,|j|\le {\rm const }|\ell| $$  unless one has $\s =-\s'$ and $|i|=|j|$.

In this last case we use  momentum conservation, if $i=j$ we get $\ell=0$  and this is explicitly excluded in our $\cup^*$. If $i=-j$ with $i=\pm 3,\pm 4$  again we get $\ell=0$ so by our previous argument $\mathfrak R_{\s,\s',\ell,i,j}^{(m)}=\emptyset$. Finally if $i=-j\neq \pm 3,\pm 4$ we obtain
$$
2 |i| \le |\pi(\ell)|\,,
$$
which obviously implies $|i|\le$ const $|\ell|$.
Let us now give a bound for the measure of a single set $\mathfrak R_{\s,\s',\ell,i,j}^{(m)}$.

By construction all the maps $\xi\to \laa^{(m)}$ are invertible  with Lipschitz inverse, so we change the variables to $\laa\in \laa^{(m)}(\OO_m)$, note that  this set is contained in $[-2 M_0,2 M_0]^4$. We have
$$
|\mathfrak R_{\s,\s',\ell,i,j}^{(m)}| \le L_m^4 \big|\{ \laa\in \laa^{(m)}(\OO_m):\quad |(\mathtt v+\e \laa)\cdot \ell +\s \tilde\Omega^{(m)}_i(\laa)+\s' \tilde\Omega^{(m)}_j(\laa)|< {\e\gamma}K_m^{-\tau} \}\big|
$$
We note that $|\tilde\Omega^{(m)}_i(\laa)|^{\rm lip}\le \e L_m M_m\le \e  |\ell|/3$, since by hypothesis $|\ell|>6M_0L_0$.
Now we can introduce an orthogonal  basis for $\R^4$ where the first basis vector is parallel to $\ell$, in this basis $$|\partial_{\laa_1}(\mathtt v+\e \laa) \cdot \ell|= \e |\ell| ,$$   so we may  estimate the measure of the resonant set as 
$$|\mathfrak R_{\s,\s',\ell,i,j}^{(m)}| \le {\rm const}\,\gamma L_0(L_0M_0)^3 |\ell|^{-1}K_m^{-\tau}$$
Hence, we estimate
\begin{multline*}
\left| \cup_m \cup^*_{\sigma\sigma'\neq0} \mathfrak R_{\s,\s',\ell,i,j}^{(m)} \right| \le \left|\bigcup_m \bigcup_{|i|\le c|\ell| ,\,|\ell|\le K_m,\, \s\s'\neq0\atop \s \mathtt r(i)+\s'\mathtt r(j)=\pi(\ell)} \mathfrak R_{\s,\s',\ell,i,j}^{(m)} \right| \\ \le {\rm const}\,\gamma L_0(L_0M_0)^3 \sum_{m=0}^\infty K_m^{-\tau +4} \le {\rm const}\,\gamma L_0(L_0M_0)^3
\end{multline*}
provided that $\tau>4$.

In the remaining cases with $\s\s'=0$ (Diophantine condition on $\laa^{(m)}$ and \emph{first Melnikov conditions}), one similiarly obtains the corresponding estimate
$$
\left| \cup_m \cup^*_{\sigma\sigma'=0} \mathfrak R_{\s,\s',\ell,i,j}^{(m)}\right|  \le {\rm const}\,\gamma L_0(L_0M_0)^3
$$
(actually, when $\s\s'=0$, this estimate is simpler to obtain and completely standard).

Hence
$$
|\OO_0\setminus \OO_\infty| \le \left|\cup_m \cup^* \mathfrak R_{\s,\s',\ell,i,j}^{(m)}\right|  \le {\rm const}\,\gamma L_0(L_0M_0)^3.$$

This research was supported by the European Research Council under
FP7, project ``Hamiltonian PDEs and small divisor problem: a dynamical systems approach'' (HamPDEs).
The second author was also supported by Programme STAR, financed by UniNA and Compagnia di San Paolo.
\end{document}